\documentclass[a4paper,12pt]{article}
\usepackage{amsmath,amsthm,amssymb,amscd}
\usepackage[all]{xy}
\usepackage{enumerate}

\theoremstyle{plain}
\newtheorem{thm}{Theorem}[section]
\newtheorem{prop}[thm]{Proposition}
\newtheorem{cor}[thm]{Corollary}
\newtheorem{lem}[thm]{Lemma}

\newtheorem{prob}[thm]{Problem}
\newtheorem{thrm}{Theorem}
\theoremstyle{definition}
\newtheorem{dfn}[thm]{Definition}
\newtheorem{exa}[thm]{Example}
\newtheorem{rem}[thm]{Remark}
\def\dim{\mathop{\mathrm{dim}}\nolimits}

\def\det{\mathop{\mathrm{det}}\nolimits}
\def\Im{\mathop{\mathrm{Im}}\nolimits}
\def\Ker{\mathop{\mathrm{Ker}}\nolimits}
\def\Coker{\mathop{\mathrm{Coker}}\nolimits}

\def\diag{\mathop{\mathrm{diag}}\nolimits}
\def\sgn{\mathop{\mathrm{sgn}}\nolimits}

\def\Cyc{\mathop{\mathrm{Cyc}}\nolimits}
\def\triv{\mathop{\mathrm{\mathbf{triv}}}\nolimits}
\def\Res{\mathop{\mathrm{Res}}\nolimits}

\newcommand{\C}{\mathbf{C}}

\newcommand{\Z}{\mathbf{Z}}
\newcommand{\Q}{\mathbf{Q}}
\newcommand{\M}{\mathcal{M}}
\newcommand{\Lg}{\mathcal{L}}
\newcommand{\A}{\mathcal{A}}
\newcommand{\h}{\mathfrak{h}}
\newcommand{\gm}{\gamma}
\newcommand{\mtau}{{\tau'_{k,\Q}}^{\hspace{-3.5mm}\M}}
\def\triv{\mathop{\mathrm{\mb{triv}}}\nolimits}
\def\Cyc{\mathop{\mathrm{Cyc}}\nolimits}
\def\pr{\mathop{\mathrm{pr}}\nolimits}
\def\cont{\mathop{\mathrm{cont}}\nolimits}
\def\LR{\mathop{\mathrm{LR}}\nolimits}
\def\GL{\mathop{\mathrm{GL}}\nolimits}
\def\Sp{\mathop{\mathrm{Sp}}\nolimits}
\def\Mob{\mathop{\text{M\"{o}b}}\nolimits}
\def\maj{\mathop{\mathrm{maj}}\nolimits}
\newcommand{\isoto}[1][]%
{{\mathop{\buildrel{\sim}\over\longrightarrow}\limits_{#1}}}
\newcommand{\mf}[1]{{\mathfrak{#1}}}
\newcommand{\mb}[1]{{\mathbf{#1}}}
\newcommand{\bb}[1]{{\mathbb{#1}}}

\title{\textbf{New series in the Johnson cokernels of the mapping class groups of surfaces}}
\author{
Naoya Enomoto\footnote{Department of Mathematics, Kyoto University, e-mail: enomoto@math.kyoto-u.ac.jp}
 \ and \ Takao Satoh\footnote{Department of Mathematics, Kyoto University, e-mail: takao@math.kyoto-u.ac.jp}
}
\date{\empty}
\begin{document}
\maketitle
\vspace{-2em}
\begin{center}
{\large \textsf{Dedicated to the memory of Midori Kato}}
\end{center}
\begin{abstract}
Let $\Sigma_{g,1}$ be a compact oriented surface of genus $g$ with one boundary component, and $\M_{g,1}$ its mapping class group.
Morita showed that the image of the $k$-th Johnson homomorphism $\tau_k^{\M}$ of $\M_{g,1}$ is contained in
the kernel $\h_{g,1}(k)$ of an $\Sp$-equivariant surjective homomorphism $H \otimes_\Z \Lg_{2g}(k+1) \to \Lg_{2g}(k+2)$,
where $H:= H_1(\Sigma_{g,1},\Z)$ and $\Lg_{2g}(k)$ is the degree $k$-part of the free Lie algebra $\Lg_{2g}$ generated by $H$.

In this paper, we study the $\Sp$-module structure of the cokernel $\h_{g,1}^{\Q}(k)/\mathrm{Im}(\tau_{k,\Q}^{\M})$ of the rational Johnson homomorphism
$\tau_{k,\Q}^{\M} := \tau_k^{\M} \otimes \mathrm{id}_{\Q}$ where $\h_{g,1}^{\Q}(k):= \h_{g,1}(k) \otimes_{\Z} \Q$.
In particular, we show that the irreducible $\Sp$-module corresponding to a partition $[1^k]$ appears in the $k$-th Johnson cokernel
for any $k \equiv 1$ $\pmod{4}$ and $k \geq 5$ with multiplicity one. We also give a new proof of the fact due to Morita that
the irreducible $\Sp$-module corresponding to a partition $[k]$ appears in the Johnson cokernel with multiplicity one for odd $k \geq 3$.

The strategy of the paper is to give explicit descriptions of maximal vectors with highest weight $[1^k]$ and $[k]$ in the Johnson cokernel.
Our construction is inspired by the Brauer-Schur-Weyl duality between $\Sp(2g,\Q)$ and the Brauer algebras, and our previous work
for the Johnson cokernel of the automorphism group of a free group.
\end{abstract}

\tableofcontents
\newpage

\section{Introduction}
Dennis Johnson established a new remarkable method to investigate
the group structure of the mapping class group of a surface and the Torelli group
in a series of his pioneer works \cite{Jo1}, \cite{Jo2}, \cite{Jo3} and \cite{Jo4} in 1980's.
Especially, he gave a finite set of generators of the Torelli group, and constructed a homomorphism $\tau$
to determine the abelianization of the Torelli group.
Now, his homomorphism $\tau$ is called the first Johnson homomorphism, and it
is generalized to the Johnson homomorphisms of higher degrees.
Over the last two decades, the study of the Johnson homomorphisms of the mapping class group has achieved a good progress
by many authors including Morita \cite{Mo1}, Hain \cite{Hai} and so on.

\vspace{0.5em}

To put it plainly, the Johnson homomorphism are used to describe \lq\lq one by one approximations" of the Torelli group as follows.
To explain it, let us fix some notations.
For a compact oriented surface $\Sigma_{g,1}$ of genus $g$ with one boundary component, let $\M_{g,1}$ be its mapping class group.
Namely, $\M_{g,1}$ is a group of isotopy classes of orientation-preserving diffeomorphisms of $\Sigma_{g,1}$ which fix the boundary component
pointwise. The fundamental group $\pi_1(\Sigma_{g,1}, *)$ of $\Sigma_{g,1}$ is isomorphic to a free group $F_{2g}$ of rank $2g$.
In this paper we fix an isomorphism $\pi_1(\Sigma_{g,1}, *) \cong F_{2g}$.
Let $\Gamma_{2g}(k)$ be the lower central series of $F_{2g}$ beginning with $\Gamma_{2g}(1) = F_{2g}$,
and set $\mathcal{L}_{2g}(k) := \Gamma_{2g}(k)/\Gamma_{2g}(k+1)$.
For each $k \geq 1$ let $\M_{g,1}(k)$ be a normal subgroup of $\M_{g,1}$ consisting of elements which act $F_{2g}/\Gamma_{2g}(k+1)$ trivially.
Then we have a descending filtration
\[ \M_{g,1}(1) \supset \M_{g,1}(2) \supset \cdots \supset \M_{g,1}(k) \supset \cdots \]
of $\M_{g,1}$ such that the first term $\M_{g,1}(1)$ is just the Torelli group $\mathcal{I}_{g,1}$. This filtration is called the Johnson
filtration of $\M_{g,1}$.
Set $\mathrm{gr}^k(\M_{g,1}) := \M_{g,1}(k)/\M_{g,1}(k+1)$ for each $k \geq 1$. Then each of $\mathrm{gr}^k(\M_{g,1})$ is an
$\mathrm{Sp}(2g,\Z)$-equivariant free abelian group of finite rank, and they are considered as one by one approximations of the Torelli group.
Although to clarify the $\mathrm{Sp}(2g,\Z)$-module structure of each of $\mathrm{gr}^k(\M_{g,1})$ plays an important role on various studies of the
Torelli group, even to determine its rank is quite a difficult problem in general.

\vspace{0.5em}

In order to study each graded quotients $\mathrm{gr}^k(\M_{g,1})$, the Johnson homomorphisms
\[ \tau_k^{\M} : \mathrm{gr}^k(\M_{g,1}) \hookrightarrow H^* \otimes_{\Z} \mathcal{L}_{2g}(k+1) \]
of $\M_{g,1}$ are valuable tools where $H:=H_1(\Sigma_{g,1},\Z)$ and $H^* :=\mathrm{Hom}_{\Z}(H,\Z)$. Here we remark that $H^*$ is canonically
isomorphic to $H$ by the Poincar\'{e} duality.
In general, the $k$-th Johnson homomorphism is denoted by $\tau_k$ simply. In this paper, however, to distinguish the Johnson homomorphism of
the mapping class group from that of the automorphism group of a free group, we attach a subscript $\M$ to that of the mapping class group.
(See Subsection {\rmfamily \ref{Ss-John}} for details.) Since each of $\tau_k^{\M}$ is an $\Sp(2g,\Z)$-equivariant injective homomorphism,
to determine the image $\mathrm{Im}(\tau_k^{\M})$ of $\tau_k^{\M}$
is one of the most basic problems. In particular, from a representation theoretic view, it is important to clarify the irreducible
decomposition of $\mathrm{Im}(\tau_{k,\Q}^{\M})$ as an $\Sp(2g,\Q)$-module where
$\tau_{k,\Q}^{\M} := \tau_{k}^{\M} \otimes \mathrm{id}_{\Q}$. In the following, the subscript $\Q$ always means tensoring with $\Q$ over $\Z$.
Now, we have $\mathrm{Im}(\tau_1^{\M}) \cong \Lambda^3 H$ due to Johnson \cite{Jo1}.
Furthermore the $\Sp(2g,\Q)$-module structure of $\mathrm{Im}(\tau_{k,\Q}^{\M})$ are completely determined for
$1 \leq k \leq 4$. (See a table in Subsection {\rmfamily \ref{Ss-John}}.)

\vspace{0.5em}

On the other hand, Morita \cite{Mo1} began to study the Johnson images systematically, and gave many remarkable results.
Here we recall some of them. First, Morita \cite{Mo1} showed that $\mathrm{Im}(\tau_k^{\M})$ is contained in the kernel $\h_{g,1}(k)$
of $H \otimes_{\Z} \mathcal{L}_{2g}(k+1) \rightarrow \mathcal{L}_{2g}(k+2)$ for any $k \geq 2$. (See Subsection {\rmfamily \ref{Ss-John}}.)
Second, he also showed that $\mathrm{Im}(\tau_k^{\M})$ does not coincide with $\h_{g,1}(k)$ in general.
Namely, the Johnson homomorphism $\tau_{k}^{\M} : \mathrm{gr}^k(\M_{g,1}) \hookrightarrow \h_{g,1}(k)$ is not surjective in general.
More precisely, he constructed an $\Sp(2g,\Q)$-equivariant surjective homomorphisms
\[ \mathrm{Tr}_k : \h_{g,1}^{\Q}(k) \rightarrow S^k H_{\Q} \]
such that $\mathrm{Tr}_k \circ \tau_{k,\Q}^{\M} \equiv 0$ for any odd $k \geq 3$ using the Magnus representation of $\M_{g,1}$.
Here $S^k H_{\Q}$ is the symmetric tensor product of $H_{\Q}$ of degree $k$, and is isomorphic to the irreducible $\Sp(2g,\Q)$-module with highest weight $[k]$.
Hence $S^k H_{\Q}$ appears in the irreducible decomposition
of the cokernel $\mathrm{Coker}(\tau_{k,\Q}^{\M}) := \h_{g,1}^{\Q}(k)/\mathrm{Im}(\tau_{k,\Q}^{\M})$ for odd $k \geq 3$.
We should remark that throughout the paper $\mathrm{Coker}(\tau_{k,\Q}^{\M})$ denotes $\h_{g,1}^{\Q}(k)/\mathrm{Im}(\tau_{k,\Q}^{\M})$,
not $H_{\Q}^* \otimes_{\Q} \mathcal{L}_{2g}^{\Q}(k+1) / \mathrm{Im}(\tau_{k,\Q}^{\M})$.
Now, the map $\mathrm{Tr}_k$ is called the Morita trace, and $S^k H_{\Q}$ the Morita obstruction. Here the term \lq\lq obstruction" means
an obstruction for the surjectivity of the Johnson homomorphism $\tau_{k,\Q}^{\M}$.
We also remark that Hiroaki Nakamura, partially Asada and Nakamura \cite{AN},
showed that the multiplicity of $S^k H_{\Q}$ in $\mathrm{Coker}(\tau_{k,\Q}^{\M})$ is exactly one in his unpublished work.

\vspace{0.5em}

From results for the irreducible decomposition of $\mathrm{Coker}(\tau_{k,\Q}^{\M})$ for low degrees, 
it seems that the number of the irreducible components in $\mathrm{Coker}(\tau_{k,\Q}^{\M})$ grows rapidly as degree increases.
At the present stage, however, there are few results for obstructions other than the Morita obstruction for a general degree $k$.
Thus, to establish a new method to detect a non-trivial irreducible component in $\mathrm{Coker}(\tau_k^{\M})$ other than the Morita obstruction
is an important problem in the study of the Johnson homomorphisms.

\vspace{0.5em}

The main purpose of the paper is to detect new series of obstructions in the Johnson cokernels. 
To state our theorem, we will use the following notations. First, we remark that for each $k \geq 1$
the symmetric group $\mf{S}_{k+2}$ of degree $k+2$ naturally acts on
the space $H_\Q^{\otimes{k+2}}$ from the right as a permutation of the components.
For each $1 \le i \le k+1$, denote by $s_i \in \mf{S}_{k+2}$
the adjacent transposition between $i$ and $i+1$, and by $\sigma_{k+2}$ the cyclic permutation $s_{k+1}s_k \cdots s_2s_1$.
Let $P$ be a subgroup of $\mf{S}_{k+2}$ which fixes $1$. The group $P$ is isomorphic to $\mf{S}_{k+1}$.
The Dynkin-Specht-Wever element $\theta_P$ for $P$ in the group algebra $\Q\mf{S}_{k+2}$ is defined to be
\[
\theta_P:=(1-s_2)(1-s_3s_2) \cdots (1-s_{k+1}s_k \cdots s_2).
\]
Our main theorem is
\begin{thrm}($=$ Theorem {\rmfamily \ref{mt2}}.)
Suppose $k \equiv 1 \ (\mathrm{mod} \, 4)$, $k \geq 5$ and $g \ge k+2$. An element 
\begin{eqnarray*}
\varphi_{[1^k]}&:=&(\omega \otimes (e_1\wedge \cdots \wedge e_k)) \cdot \theta_P \cdot (1+\sigma_{k+2}+ \cdots +\sigma_{k+2}^{k+1})
\end{eqnarray*}
is an $\Sp$-maximal vector of weight $[1^k]$ in $\mf{h}_{g,1}^{\Q}(k)$.
Moreover this gives a unique $\Sp$-irreducible component with highest weight $[1^k]$ in $\Coker{\tau_{k,\Q}^{\M}}$.
\end{thrm}
In addition to this, we also give a new proof of the fact that the Morita obstruction uniquely appears in $\mathrm{Coker}(\tau_k^{\M})$ for odd $k \geq 3$,
due to Morita \cite{Mo1} and Nakamura. (See Theorem {\rmfamily \ref{mt1}}.)

\vspace{0.5em}

In order to prove these, we use two key facts. The first one is a remarkable work with respect to $\mathrm{gr}^k(\M_{g,1})$ due to Hain \cite{Hai}.
In general, the graded sum $\mathrm{gr}(\M_{g,1}) := \oplus_{k \geq 1} \mathrm{gr}^k(\M_{g,1})$ has a Lie algebra structure induced from the commutator
bracket of $\mathcal{I}_{g,1}$.
In \cite{Hai}, Hain showed that the Lie algebra $\mathrm{gr}_{\Q}(\M_{g,1})$ is generated by the degree one part $\mathrm{gr}_{\Q}^1(\M_{g,1})$
as a Lie algebra. This shows the following. Let $\M_{g,1}'(k)$ be the lower central series of $\mathcal{I}_{g,1}$ and set
$\mathrm{gr}^k(\M_{g,1}') := \M_{g,1}'(k)/\M_{g,1}'(k+1)$. Then we can define the Johnson homomorphism
like homomorphism
\[ {\tau'_{k}}^{\M} : \mathrm{gr}^k(\M_{g,1}') \rightarrow \h_{g,1}(k). \]
(See Subsection {\rmfamily \ref{Ss-John}}.) Then Hain's result above induces $\mathrm{Im}(\tau_{k,\Q}^{\M})= \Im(\mtau)$ for any $k \geq 1$.

\vspace{0.5em}

The second is our previous result for the cokernel of the Johnson homomorphism of the automorphism group of a free group.
By a classical work of Dehn and Nielsen, it is known that a natural homomorphism $\M_{g,1} \rightarrow \mathrm{Aut}\,F_{2g}$ induced from the action
of $\M_{g,1}$ of the fundamental group $\pi_1(\Sigma_{g,1},*) \cong F_{2g}$ is injective. Namely, we can consider $\M_{g,1}$ as a subgroup of
$\mathrm{Aut}\,F_{2g}$. From this view point, we can apply results for the Johnson homomorphisms of $\mathrm{Aut}\,F_{2g}$ to the study of that of $\M_{g,1}$.
For any $n \geq 2$, in general, a subgroup $\mathrm{IA}_n$ consisting of automorphisms of a free group $F_n$ which acts on $H_1(F_n,\Z)$ trivially is called
the IA-automorphism group of $F_n$. Let $\mathcal{A}_n'(k)$ be the lower central series of $\mathrm{IA}_n$, and set
$\mathrm{gr}^k(\mathcal{A}_n') := \mathcal{A}_n'(k)/\mathcal{A}_n'(k+1)$ for any $k \geq 1$. Then we can define the Johnson homomorphism
$\tau_k' : \mathrm{gr}^k(\mathcal{A}_n') \rightarrow H^* \otimes_{\Z} \mathcal{L}_n(k+1)$ for each $k \geq 1$.
Then, in our paper \cite{S11}, we showed that for $k \geq 2$ and $n \geq k+2$,
\[ \mathrm{Coker}(\tau_{k,\Q}') \cong \mathcal{C}_n^{\Q}(k) \]
where $\mathcal{C}_n(k) := H^{\otimes k} / \langle a_1 \otimes \cdots \otimes a_k - a_2 \otimes \cdots \otimes a_k \otimes a_1 \,|\, a_i \in H \rangle$.
(See Subsection {\rmfamily \ref{Ss-John}} for details.)

\vspace{0.5em}

In our previous paper \cite{ES}, we gave the irreducible decomposition of $\mathrm{Coker}(\tau_{k,\Q}') \cong \mathcal{C}_n^{\Q}(k)$
as a $\mathrm{GL}(n,\Q)$-module.
Especially, we showed that $S^k H_{\Q}$, which is also called the Morita obstruction, appears in $\mathrm{Coker}(\tau_{k,\Q}')$ with multiplicity one
for any $k \geq 2$, and that $\Lambda^k H_{\Q}$ appears with multiplicity one for odd $k \geq 3$.

\vspace{0.5em}

We remark that, as a $\mathrm{GL}(n,\Q)$-module, $\mathcal{C}_n^{\Q}(k)$ is isomorphic to the invariant part $a_n(k):=(H_{\Q}^{\otimes k})^{\Cyc_k}$ of $H_{\Q}^{\otimes k}$
by the action of $\Cyc_k$. Namely, the cokernel $\mathrm{Coker}(\tau_{k, \Q}')$ is isomorphic to Kontsevich's
$a_n(k)$ as a $\mathrm{GL}(n,\Q)$-module. We also remark that in our notation $a_n(k)$ is considered for any $n \geq 2$ in constrast to Kontsevich's notation
for even $n=2g$. (See \cite{Kon1} and \cite{Kon2}.)

\vspace{0.5em}

Combining Hain's result above and the fact $\mathrm{Coker}(\tau_{k,\Q}') \cong \mathcal{C}_n^{\Q}(k)$ for $n \geq k+2$, we can establish a new method to
detect non-trivial $\Sp$-irreducible components in $\mathrm{Coker}(\tau_k^{\M})$. (For more details, see Section {\ref{sec:str}}.)
The present paper produces the first successful results for the use of such method.

\vspace{0.5em}

{\bf Note added}:
After we wrote this paper, Professor Hiroaki Nakamura told us about the following
personal communication. 
In 1996, in his mail to Professor Shigeyuki Morita, he mentioned that, for $1 \leq k \leq 3$, an $\Sp$-module $[1^{4k+1}]$ appears in $\mf{h}_{g,1}(k)$ with multiplicity one,
based on his explicit calculation in \cite{NT}. And he conjectured that these
$\Sp$-irreducible components $[1^{4k+1}]$ survive in the Johnson
cokernel.

\section{Notations}

Throughout the paper, we use the following notations. Let $G$ be a group and $N$ a normal subgroup of $G$.
\begin{itemize}
\item The binomial coefficient $\binom{n}{r}$ is denoted by ${}_{n}C_{r}$.
\item For any real number $x$, we set $\lfloor x \rfloor := \mathrm{max} \{ n \in \Z \,|\, n \leq x \}$.
\item For any integer $p$, set
\[ \delta_{p \equiv a \, (\mathrm{mod}\, m)} := \begin{cases}
                                                 1 \hspace{1em} & \mathrm{if} \hspace{1em} p \equiv a \, (\mathrm{mod}\, m), \\
                                                 0 \hspace{1em} & \mathrm{if} \hspace{1em} \mathrm{otherwise}.
                                              \end{cases}\]
\item The automorphism group $\mathrm{Aut}\,F_n$ of $F_n$ acts on $F_n$ from the right unless otherwise noted.
      For any $\sigma \in \mathrm{Aut}\,F_n$ and $x \in F_n$, the action of $\sigma$ on $x$ is denoted by $x^{\sigma}$.
\item For an element $g \in G$, we also denote the coset class of $g$ by $g \in G/N$ if there is no confusion.
\item For elements $x$ and $y$ of $G$, the commutator bracket $[x,y]$ of $x$ and $y$
      is defined to be $[x,y]:=xyx^{-1}y^{-1}$.
\item For elements $g_1, \ldots, g_k \in G$, a left-normed commutator
\[ [[ \cdots [[ g_{1},g_{2}],g_{3}], \cdots ], g_{k}] \]
of weight $k$ is denoted by $[g_{i_1},g_{i_2}, \cdots, g_{i_k}]$.
\item For any $\Z$-module $M$ and a commutative ring $R$, we denote $M \otimes_{\Z} R$ by the symbol obtained by attaching a subscript $R$ to $M$, like
      $M_{R}$ or $M^{R}$. Similarly, for any $\Z$-linear map $f: A \rightarrow B$,
      the induced $R$-linear map $A_{R} \rightarrow B_{R}$ is denoted by $f_{R}$
      or $f^{R}$.
\item For a semisimple $G$-module $M$ and an irreducible $G$-module $N$, we denote by $[N:M]$ the multiplicity of $N$ in the irreducible decomposition
of $M$.
\end{itemize}

\section{Johnson homomorphisms of the mapping class groups and the automorphism group of free groups}

\subsection{Mapping class groups of surfaces}

Here we recall some properties of the mapping class groups of surfaces.
For any integer $g \geq 1$, let $\Sigma_{g,1}$ be the compact oriented surface of genus $g$ with one boundary component.
We denote by $\M_{g,1}$ the mapping class group of ${\Sigma}_{g,1}$. Namely, $\M_{g,1}$ is the group of isotopy classes of
orientation preserving diffeomorphisms of ${\Sigma}_{g,1}$ which fix the boundary pointwise. \\
\quad The mapping class group $\M_{g,1}$ has an important normal subgroup called the Torelli group.
Let $\mu_{\M} : \M_{g,1} \rightarrow \mathrm{Aut}(H_1(\Sigma_{g,1},\Z))$ be the classical representation of $\M_{g,1}$
induced from the action of $\M_{g,1}$ on the integral first homology group $H_1(\Sigma_{g,1},\Z)$ of $\Sigma_{g,1}$.
The kernel of $\mu_{\M}$ is called the Torelli group, denoted by $\mathcal{I}_{g,1}$. Namely, $\mathcal{I}_{g,1}$ consists of mapping
classes of $\Sigma_{g,1}$ which act on $H_1(\Sigma_{g,1},\Z)$ trivially. \\
\quad Let us observe the image of $\mu_{\M}$. Take a base point $*$ of $\Sigma_{g,1}$ on the boundary.
Then the fundamental group $\pi_1(\Sigma_{g,1}, *)$ of $\Sigma_{g,1}$ is a free group of rank $2g$.
We fix a basis $x_1, \ldots, x_{2g}$ of $\pi_1(\Sigma_{g,1},*)$ as shown Figure 1.

\begin{figure}[h]
\[
\unitlength 0.1in
\begin{picture}( 55.0000, 18.0000)(  9.0000,-22.0000)
%
\special{pn 8}%
\special{pa 1800 400}%
\special{pa 6200 400}%
\special{fp}%
%
\special{pn 8}%
\special{pa 1800 2200}%
\special{pa 6200 2200}%
\special{fp}%
%
\special{pn 8}%
\special{ar 1800 1300 900 900  1.5707963 4.7123890}%
%
\special{pn 8}%
\special{ar 1760 1270 160 270  1.5707963 4.7123890}%
%
\special{pn 8}%
\special{ar 1660 1260 80 190  4.7123890 6.2831853}%
\special{ar 1660 1260 80 190  0.0000000 1.5707963}%
%
\special{pn 8}%
\special{ar 2760 1270 160 270  1.5707963 4.7123890}%
%
\special{pn 8}%
\special{ar 2660 1260 80 190  4.7123890 6.2831853}%
\special{ar 2660 1260 80 190  0.0000000 1.5707963}%
%
\special{pn 8}%
\special{ar 4960 1270 160 270  1.5707963 4.7123890}%
%
\special{pn 8}%
\special{ar 4860 1260 80 190  4.7123890 6.2831853}%
\special{ar 4860 1260 80 190  0.0000000 1.5707963}%
%
\special{pn 13}%
\special{ar 2040 1170 220 770  4.7123890 6.2831853}%
\special{ar 2040 1170 220 770  0.0000000 1.5707963}%
%
\special{pn 13}%
\special{pa 2060 1940}%
\special{pa 6340 1940}%
\special{fp}%
%
\special{pn 13}%
\special{ar 3070 1170 220 770  4.7123890 6.2831853}%
\special{ar 3070 1170 220 770  0.0000000 1.5707963}%
%
\special{pn 13}%
\special{ar 5240 1170 220 770  4.7123890 6.2831853}%
\special{ar 5240 1170 220 770  0.0000000 1.5707963}%
%
\special{pn 13}%
\special{pa 2040 1940}%
\special{pa 2014 1924}%
\special{pa 1986 1906}%
\special{pa 1960 1890}%
\special{pa 1932 1872}%
\special{pa 1906 1854}%
\special{pa 1880 1836}%
\special{pa 1852 1818}%
\special{pa 1826 1800}%
\special{pa 1802 1780}%
\special{pa 1776 1760}%
\special{pa 1750 1740}%
\special{pa 1726 1720}%
\special{pa 1702 1698}%
\special{pa 1678 1676}%
\special{pa 1656 1654}%
\special{pa 1634 1630}%
\special{pa 1612 1606}%
\special{pa 1592 1582}%
\special{pa 1574 1556}%
\special{pa 1556 1528}%
\special{pa 1538 1500}%
\special{pa 1524 1472}%
\special{pa 1510 1442}%
\special{pa 1496 1412}%
\special{pa 1486 1380}%
\special{pa 1476 1348}%
\special{pa 1470 1314}%
\special{pa 1464 1280}%
\special{pa 1460 1244}%
\special{pa 1460 1206}%
\special{pa 1460 1168}%
\special{pa 1462 1130}%
\special{pa 1468 1092}%
\special{pa 1474 1054}%
\special{pa 1484 1018}%
\special{pa 1494 982}%
\special{pa 1508 950}%
\special{pa 1522 918}%
\special{pa 1540 890}%
\special{pa 1560 866}%
\special{pa 1580 844}%
\special{pa 1604 828}%
\special{pa 1630 816}%
\special{pa 1658 810}%
\special{pa 1686 808}%
\special{pa 1718 810}%
\special{pa 1750 818}%
\special{pa 1780 830}%
\special{pa 1812 846}%
\special{pa 1842 864}%
\special{pa 1870 884}%
\special{pa 1896 906}%
\special{pa 1920 930}%
\special{pa 1940 956}%
\special{pa 1956 984}%
\special{pa 1972 1012}%
\special{pa 1982 1040}%
\special{pa 1992 1070}%
\special{pa 2000 1100}%
\special{pa 2006 1132}%
\special{pa 2008 1164}%
\special{pa 2012 1198}%
\special{pa 2012 1230}%
\special{pa 2012 1264}%
\special{pa 2012 1298}%
\special{pa 2010 1332}%
\special{pa 2008 1366}%
\special{pa 2006 1400}%
\special{pa 2004 1434}%
\special{pa 2002 1468}%
\special{pa 2000 1502}%
\special{pa 2000 1536}%
\special{pa 2000 1570}%
\special{pa 2000 1602}%
\special{pa 2002 1634}%
\special{pa 2004 1666}%
\special{pa 2006 1698}%
\special{pa 2008 1730}%
\special{pa 2010 1762}%
\special{pa 2014 1794}%
\special{pa 2016 1824}%
\special{pa 2020 1856}%
\special{pa 2024 1888}%
\special{pa 2028 1918}%
\special{pa 2030 1940}%
\special{sp}%
%
\special{pn 13}%
\special{pa 3060 1940}%
\special{pa 3034 1924}%
\special{pa 3006 1906}%
\special{pa 2980 1890}%
\special{pa 2952 1872}%
\special{pa 2926 1854}%
\special{pa 2900 1836}%
\special{pa 2872 1818}%
\special{pa 2846 1800}%
\special{pa 2822 1780}%
\special{pa 2796 1760}%
\special{pa 2770 1740}%
\special{pa 2746 1720}%
\special{pa 2722 1698}%
\special{pa 2698 1676}%
\special{pa 2676 1654}%
\special{pa 2654 1630}%
\special{pa 2632 1606}%
\special{pa 2612 1582}%
\special{pa 2594 1556}%
\special{pa 2576 1528}%
\special{pa 2558 1500}%
\special{pa 2544 1472}%
\special{pa 2530 1442}%
\special{pa 2516 1412}%
\special{pa 2506 1380}%
\special{pa 2496 1348}%
\special{pa 2490 1314}%
\special{pa 2484 1280}%
\special{pa 2480 1244}%
\special{pa 2480 1206}%
\special{pa 2480 1168}%
\special{pa 2482 1130}%
\special{pa 2488 1092}%
\special{pa 2494 1054}%
\special{pa 2504 1018}%
\special{pa 2514 982}%
\special{pa 2528 950}%
\special{pa 2542 918}%
\special{pa 2560 890}%
\special{pa 2580 866}%
\special{pa 2600 844}%
\special{pa 2624 828}%
\special{pa 2650 816}%
\special{pa 2678 810}%
\special{pa 2706 808}%
\special{pa 2738 810}%
\special{pa 2770 818}%
\special{pa 2800 830}%
\special{pa 2832 846}%
\special{pa 2862 864}%
\special{pa 2890 884}%
\special{pa 2916 906}%
\special{pa 2940 930}%
\special{pa 2960 956}%
\special{pa 2976 984}%
\special{pa 2992 1012}%
\special{pa 3002 1040}%
\special{pa 3012 1070}%
\special{pa 3020 1100}%
\special{pa 3026 1132}%
\special{pa 3028 1164}%
\special{pa 3032 1198}%
\special{pa 3032 1230}%
\special{pa 3032 1264}%
\special{pa 3032 1298}%
\special{pa 3030 1332}%
\special{pa 3028 1366}%
\special{pa 3026 1400}%
\special{pa 3024 1434}%
\special{pa 3022 1468}%
\special{pa 3020 1502}%
\special{pa 3020 1536}%
\special{pa 3020 1570}%
\special{pa 3020 1602}%
\special{pa 3022 1634}%
\special{pa 3024 1666}%
\special{pa 3026 1698}%
\special{pa 3028 1730}%
\special{pa 3030 1762}%
\special{pa 3034 1794}%
\special{pa 3036 1824}%
\special{pa 3040 1856}%
\special{pa 3044 1888}%
\special{pa 3048 1918}%
\special{pa 3050 1940}%
\special{sp}%
%
\special{pn 13}%
\special{pa 5240 1940}%
\special{pa 5214 1924}%
\special{pa 5186 1906}%
\special{pa 5160 1890}%
\special{pa 5132 1872}%
\special{pa 5106 1854}%
\special{pa 5080 1836}%
\special{pa 5052 1818}%
\special{pa 5026 1800}%
\special{pa 5002 1780}%
\special{pa 4976 1760}%
\special{pa 4950 1740}%
\special{pa 4926 1720}%
\special{pa 4902 1698}%
\special{pa 4878 1676}%
\special{pa 4856 1654}%
\special{pa 4834 1630}%
\special{pa 4812 1606}%
\special{pa 4792 1582}%
\special{pa 4774 1556}%
\special{pa 4756 1528}%
\special{pa 4738 1500}%
\special{pa 4724 1472}%
\special{pa 4710 1442}%
\special{pa 4696 1412}%
\special{pa 4686 1380}%
\special{pa 4676 1348}%
\special{pa 4670 1314}%
\special{pa 4664 1280}%
\special{pa 4660 1244}%
\special{pa 4660 1206}%
\special{pa 4660 1168}%
\special{pa 4662 1130}%
\special{pa 4668 1092}%
\special{pa 4674 1054}%
\special{pa 4684 1018}%
\special{pa 4694 982}%
\special{pa 4708 950}%
\special{pa 4722 918}%
\special{pa 4740 890}%
\special{pa 4760 866}%
\special{pa 4780 844}%
\special{pa 4804 828}%
\special{pa 4830 816}%
\special{pa 4858 810}%
\special{pa 4886 808}%
\special{pa 4918 810}%
\special{pa 4950 818}%
\special{pa 4980 830}%
\special{pa 5012 846}%
\special{pa 5042 864}%
\special{pa 5070 884}%
\special{pa 5096 906}%
\special{pa 5120 930}%
\special{pa 5140 956}%
\special{pa 5156 984}%
\special{pa 5172 1012}%
\special{pa 5182 1040}%
\special{pa 5192 1070}%
\special{pa 5200 1100}%
\special{pa 5206 1132}%
\special{pa 5208 1164}%
\special{pa 5212 1198}%
\special{pa 5212 1230}%
\special{pa 5212 1264}%
\special{pa 5212 1298}%
\special{pa 5210 1332}%
\special{pa 5208 1366}%
\special{pa 5206 1400}%
\special{pa 5204 1434}%
\special{pa 5202 1468}%
\special{pa 5200 1502}%
\special{pa 5200 1536}%
\special{pa 5200 1570}%
\special{pa 5200 1602}%
\special{pa 5202 1634}%
\special{pa 5204 1666}%
\special{pa 5206 1698}%
\special{pa 5208 1730}%
\special{pa 5210 1762}%
\special{pa 5214 1794}%
\special{pa 5216 1824}%
\special{pa 5220 1856}%
\special{pa 5224 1888}%
\special{pa 5228 1918}%
\special{pa 5230 1940}%
\special{sp}%
\put(16.5000,-17.1000){\makebox(0,0)[rt]{$x_1$}}%
\put(27.4000,-17.4000){\makebox(0,0)[rt]{$x_2$}}%
\put(48.8000,-17.1000){\makebox(0,0)[rt]{$x_g$}}%
\put(54.5000,-6.6000){\makebox(0,0)[lb]{$x_{g+1}$}}%
\put(32.5000,-6.8000){\makebox(0,0)[lb]{$x_{2g-1}$}}%
\put(22.4000,-6.7000){\makebox(0,0)[lb]{$x_{2g}$}}%
%
\special{pn 8}%
\special{ar 6200 1300 200 900  4.7123890 6.2831853}%
\special{ar 6200 1300 200 900  0.0000000 1.5707963}%
%
\special{pn 8}%
\special{ar 6200 1300 200 900  1.5707963 1.5926145}%
\special{ar 6200 1300 200 900  1.6580691 1.6798872}%
\special{ar 6200 1300 200 900  1.7453418 1.7671600}%
\special{ar 6200 1300 200 900  1.8326145 1.8544327}%
\special{ar 6200 1300 200 900  1.9198872 1.9417054}%
\special{ar 6200 1300 200 900  2.0071600 2.0289781}%
\special{ar 6200 1300 200 900  2.0944327 2.1162509}%
\special{ar 6200 1300 200 900  2.1817054 2.2035236}%
\special{ar 6200 1300 200 900  2.2689781 2.2907963}%
\special{ar 6200 1300 200 900  2.3562509 2.3780691}%
\special{ar 6200 1300 200 900  2.4435236 2.4653418}%
\special{ar 6200 1300 200 900  2.5307963 2.5526145}%
\special{ar 6200 1300 200 900  2.6180691 2.6398872}%
\special{ar 6200 1300 200 900  2.7053418 2.7271600}%
\special{ar 6200 1300 200 900  2.7926145 2.8144327}%
\special{ar 6200 1300 200 900  2.8798872 2.9017054}%
\special{ar 6200 1300 200 900  2.9671600 2.9889781}%
\special{ar 6200 1300 200 900  3.0544327 3.0762509}%
\special{ar 6200 1300 200 900  3.1417054 3.1635236}%
\special{ar 6200 1300 200 900  3.2289781 3.2507963}%
\special{ar 6200 1300 200 900  3.3162509 3.3380691}%
\special{ar 6200 1300 200 900  3.4035236 3.4253418}%
\special{ar 6200 1300 200 900  3.4907963 3.5126145}%
\special{ar 6200 1300 200 900  3.5780691 3.5998872}%
\special{ar 6200 1300 200 900  3.6653418 3.6871600}%
\special{ar 6200 1300 200 900  3.7526145 3.7744327}%
\special{ar 6200 1300 200 900  3.8398872 3.8617054}%
\special{ar 6200 1300 200 900  3.9271600 3.9489781}%
\special{ar 6200 1300 200 900  4.0144327 4.0362509}%
\special{ar 6200 1300 200 900  4.1017054 4.1235236}%
\special{ar 6200 1300 200 900  4.1889781 4.2107963}%
\special{ar 6200 1300 200 900  4.2762509 4.2980691}%
\special{ar 6200 1300 200 900  4.3635236 4.3853418}%
\special{ar 6200 1300 200 900  4.4507963 4.4726145}%
\special{ar 6200 1300 200 900  4.5380691 4.5598872}%
\special{ar 6200 1300 200 900  4.6253418 4.6471600}%
%
\special{pn 13}%
\special{pa 2030 1920}%
\special{pa 2020 1890}%
\special{pa 2010 1860}%
\special{pa 2000 1830}%
\special{pa 1988 1798}%
\special{pa 1978 1768}%
\special{pa 1966 1738}%
\special{pa 1954 1710}%
\special{pa 1940 1680}%
\special{pa 1928 1650}%
\special{pa 1912 1622}%
\special{pa 1898 1594}%
\special{pa 1880 1566}%
\special{pa 1864 1540}%
\special{pa 1844 1514}%
\special{pa 1826 1486}%
\special{pa 1808 1460}%
\special{pa 1790 1432}%
\special{pa 1776 1406}%
\special{pa 1762 1376}%
\special{pa 1752 1346}%
\special{pa 1746 1314}%
\special{pa 1740 1290}%
\special{sp}%
%
\special{pn 13}%
\special{pa 1740 1290}%
\special{pa 1746 1258}%
\special{pa 1752 1228}%
\special{pa 1758 1196}%
\special{pa 1768 1166}%
\special{pa 1778 1136}%
\special{pa 1786 1104}%
\special{pa 1794 1072}%
\special{pa 1802 1038}%
\special{pa 1806 1004}%
\special{pa 1812 968}%
\special{pa 1816 930}%
\special{pa 1818 894}%
\special{pa 1822 856}%
\special{pa 1824 818}%
\special{pa 1828 782}%
\special{pa 1830 744}%
\special{pa 1834 708}%
\special{pa 1840 674}%
\special{pa 1846 640}%
\special{pa 1852 606}%
\special{pa 1860 576}%
\special{pa 1872 546}%
\special{pa 1884 520}%
\special{pa 1898 494}%
\special{pa 1914 472}%
\special{pa 1934 454}%
\special{pa 1956 436}%
\special{pa 1982 424}%
\special{pa 2010 414}%
\special{pa 2042 408}%
\special{pa 2074 402}%
\special{pa 2080 400}%
\special{sp -0.045}%
%
\special{pn 13}%
\special{pa 2750 1290}%
\special{pa 2756 1258}%
\special{pa 2762 1228}%
\special{pa 2768 1196}%
\special{pa 2778 1166}%
\special{pa 2788 1136}%
\special{pa 2796 1104}%
\special{pa 2804 1072}%
\special{pa 2812 1038}%
\special{pa 2816 1004}%
\special{pa 2822 968}%
\special{pa 2826 930}%
\special{pa 2828 894}%
\special{pa 2832 856}%
\special{pa 2834 818}%
\special{pa 2838 782}%
\special{pa 2840 744}%
\special{pa 2844 708}%
\special{pa 2850 674}%
\special{pa 2856 640}%
\special{pa 2862 606}%
\special{pa 2870 576}%
\special{pa 2882 546}%
\special{pa 2894 520}%
\special{pa 2908 494}%
\special{pa 2924 472}%
\special{pa 2944 454}%
\special{pa 2966 436}%
\special{pa 2992 424}%
\special{pa 3020 414}%
\special{pa 3052 408}%
\special{pa 3084 402}%
\special{pa 3090 400}%
\special{sp -0.045}%
%
\special{pn 13}%
\special{pa 3030 1920}%
\special{pa 3020 1890}%
\special{pa 3010 1860}%
\special{pa 3000 1830}%
\special{pa 2988 1798}%
\special{pa 2978 1768}%
\special{pa 2966 1738}%
\special{pa 2954 1710}%
\special{pa 2940 1680}%
\special{pa 2928 1650}%
\special{pa 2912 1622}%
\special{pa 2898 1594}%
\special{pa 2880 1566}%
\special{pa 2864 1540}%
\special{pa 2844 1514}%
\special{pa 2826 1486}%
\special{pa 2808 1460}%
\special{pa 2790 1432}%
\special{pa 2776 1406}%
\special{pa 2762 1376}%
\special{pa 2752 1346}%
\special{pa 2746 1314}%
\special{pa 2740 1290}%
\special{sp}%
%
\special{pn 13}%
\special{pa 5230 1930}%
\special{pa 5220 1900}%
\special{pa 5210 1870}%
\special{pa 5200 1840}%
\special{pa 5188 1808}%
\special{pa 5178 1778}%
\special{pa 5166 1748}%
\special{pa 5154 1720}%
\special{pa 5140 1690}%
\special{pa 5128 1660}%
\special{pa 5112 1632}%
\special{pa 5098 1604}%
\special{pa 5080 1576}%
\special{pa 5064 1550}%
\special{pa 5044 1524}%
\special{pa 5026 1496}%
\special{pa 5008 1470}%
\special{pa 4990 1442}%
\special{pa 4976 1416}%
\special{pa 4962 1386}%
\special{pa 4952 1356}%
\special{pa 4946 1324}%
\special{pa 4940 1300}%
\special{sp}%
%
\special{pn 13}%
\special{pa 4950 1290}%
\special{pa 4956 1258}%
\special{pa 4962 1228}%
\special{pa 4968 1196}%
\special{pa 4978 1166}%
\special{pa 4988 1136}%
\special{pa 4996 1104}%
\special{pa 5004 1072}%
\special{pa 5012 1038}%
\special{pa 5016 1004}%
\special{pa 5022 968}%
\special{pa 5026 930}%
\special{pa 5028 894}%
\special{pa 5032 856}%
\special{pa 5034 818}%
\special{pa 5038 782}%
\special{pa 5040 744}%
\special{pa 5044 708}%
\special{pa 5050 674}%
\special{pa 5056 640}%
\special{pa 5062 606}%
\special{pa 5070 576}%
\special{pa 5082 546}%
\special{pa 5094 520}%
\special{pa 5108 494}%
\special{pa 5124 472}%
\special{pa 5144 454}%
\special{pa 5166 436}%
\special{pa 5192 424}%
\special{pa 5220 414}%
\special{pa 5252 408}%
\special{pa 5284 402}%
\special{pa 5290 400}%
\special{sp -0.045}%
\put(41.1000,-12.7000){\makebox(0,0){$\cdots$}}%
\put(63.8000,-18.8000){\makebox(0,0)[lt]{$*$}}%
%
\special{pn 8}%
\special{pa 2260 1130}%
\special{pa 2260 1050}%
\special{fp}%
\special{sh 1}%
\special{pa 2260 1050}%
\special{pa 2240 1118}%
\special{pa 2260 1104}%
\special{pa 2280 1118}%
\special{pa 2260 1050}%
\special{fp}%
%
\special{pn 8}%
\special{pa 1460 1220}%
\special{pa 1460 1170}%
\special{fp}%
\special{sh 1}%
\special{pa 1460 1170}%
\special{pa 1440 1238}%
\special{pa 1460 1224}%
\special{pa 1480 1238}%
\special{pa 1460 1170}%
\special{fp}%
%
\special{pn 8}%
\special{pa 2480 1200}%
\special{pa 2480 1150}%
\special{fp}%
\special{sh 1}%
\special{pa 2480 1150}%
\special{pa 2460 1218}%
\special{pa 2480 1204}%
\special{pa 2500 1218}%
\special{pa 2480 1150}%
\special{fp}%
%
\special{pn 8}%
\special{pa 3290 1140}%
\special{pa 3290 1060}%
\special{fp}%
\special{sh 1}%
\special{pa 3290 1060}%
\special{pa 3270 1128}%
\special{pa 3290 1114}%
\special{pa 3310 1128}%
\special{pa 3290 1060}%
\special{fp}%
%
\special{pn 8}%
\special{pa 5460 1180}%
\special{pa 5460 1100}%
\special{fp}%
\special{sh 1}%
\special{pa 5460 1100}%
\special{pa 5440 1168}%
\special{pa 5460 1154}%
\special{pa 5480 1168}%
\special{pa 5460 1100}%
\special{fp}%
%
\special{pn 8}%
\special{pa 4660 1230}%
\special{pa 4660 1180}%
\special{fp}%
\special{sh 1}%
\special{pa 4660 1180}%
\special{pa 4640 1248}%
\special{pa 4660 1234}%
\special{pa 4680 1248}%
\special{pa 4660 1180}%
\special{fp}%
\put(61.5000,-12.3000){\makebox(0,0)[lb]{$\zeta$}}%
%
\special{pn 8}%
\special{ar 5870 1300 200 900  1.5707963 1.5926145}%
\special{ar 5870 1300 200 900  1.6580691 1.6798872}%
\special{ar 5870 1300 200 900  1.7453418 1.7671600}%
\special{ar 5870 1300 200 900  1.8326145 1.8544327}%
\special{ar 5870 1300 200 900  1.9198872 1.9417054}%
\special{ar 5870 1300 200 900  2.0071600 2.0289781}%
\special{ar 5870 1300 200 900  2.0944327 2.1162509}%
\special{ar 5870 1300 200 900  2.1817054 2.2035236}%
\special{ar 5870 1300 200 900  2.2689781 2.2907963}%
\special{ar 5870 1300 200 900  2.3562509 2.3780691}%
\special{ar 5870 1300 200 900  2.4435236 2.4653418}%
\special{ar 5870 1300 200 900  2.5307963 2.5526145}%
\special{ar 5870 1300 200 900  2.6180691 2.6398872}%
\special{ar 5870 1300 200 900  2.7053418 2.7271600}%
\special{ar 5870 1300 200 900  2.7926145 2.8144327}%
\special{ar 5870 1300 200 900  2.8798872 2.9017054}%
\special{ar 5870 1300 200 900  2.9671600 2.9889781}%
\special{ar 5870 1300 200 900  3.0544327 3.0762509}%
\special{ar 5870 1300 200 900  3.1417054 3.1635236}%
\special{ar 5870 1300 200 900  3.2289781 3.2507963}%
\special{ar 5870 1300 200 900  3.3162509 3.3380691}%
\special{ar 5870 1300 200 900  3.4035236 3.4253418}%
\special{ar 5870 1300 200 900  3.4907963 3.5126145}%
\special{ar 5870 1300 200 900  3.5780691 3.5998872}%
\special{ar 5870 1300 200 900  3.6653418 3.6871600}%
\special{ar 5870 1300 200 900  3.7526145 3.7744327}%
\special{ar 5870 1300 200 900  3.8398872 3.8617054}%
\special{ar 5870 1300 200 900  3.9271600 3.9489781}%
\special{ar 5870 1300 200 900  4.0144327 4.0362509}%
\special{ar 5870 1300 200 900  4.1017054 4.1235236}%
\special{ar 5870 1300 200 900  4.1889781 4.2107963}%
\special{ar 5870 1300 200 900  4.2762509 4.2980691}%
\special{ar 5870 1300 200 900  4.3635236 4.3853418}%
\special{ar 5870 1300 200 900  4.4507963 4.4726145}%
\special{ar 5870 1300 200 900  4.5380691 4.5598872}%
\special{ar 5870 1300 200 900  4.6253418 4.6471600}%
%
\special{pn 13}%
\special{pa 6310 1940}%
\special{pa 6276 1942}%
\special{pa 6244 1946}%
\special{pa 6212 1950}%
\special{pa 6182 1958}%
\special{pa 6152 1970}%
\special{pa 6128 1986}%
\special{pa 6104 2006}%
\special{pa 6082 2028}%
\special{pa 6060 2054}%
\special{pa 6038 2078}%
\special{pa 6016 2104}%
\special{pa 5994 2126}%
\special{pa 5970 2148}%
\special{pa 5944 2166}%
\special{pa 5916 2182}%
\special{pa 5888 2198}%
\special{pa 5880 2200}%
\special{sp}%
%
\special{pn 13}%
\special{pa 5850 400}%
\special{pa 5876 420}%
\special{pa 5900 440}%
\special{pa 5922 464}%
\special{pa 5938 490}%
\special{pa 5952 520}%
\special{pa 5962 550}%
\special{pa 5970 582}%
\special{pa 5978 614}%
\special{pa 5984 648}%
\special{pa 5990 680}%
\special{pa 5996 712}%
\special{pa 6000 744}%
\special{pa 6004 776}%
\special{pa 6010 808}%
\special{pa 6014 840}%
\special{pa 6018 870}%
\special{pa 6022 902}%
\special{pa 6026 932}%
\special{pa 6028 964}%
\special{pa 6032 994}%
\special{pa 6036 1026}%
\special{pa 6038 1056}%
\special{pa 6042 1088}%
\special{pa 6044 1118}%
\special{pa 6048 1150}%
\special{pa 6052 1180}%
\special{pa 6054 1212}%
\special{pa 6058 1244}%
\special{pa 6062 1274}%
\special{pa 6066 1306}%
\special{pa 6070 1338}%
\special{pa 6074 1370}%
\special{pa 6080 1402}%
\special{pa 6084 1434}%
\special{pa 6090 1466}%
\special{pa 6094 1500}%
\special{pa 6102 1532}%
\special{pa 6108 1566}%
\special{pa 6114 1600}%
\special{pa 6122 1634}%
\special{pa 6130 1666}%
\special{pa 6140 1700}%
\special{pa 6150 1732}%
\special{pa 6162 1764}%
\special{pa 6176 1792}%
\special{pa 6192 1820}%
\special{pa 6210 1846}%
\special{pa 6232 1870}%
\special{pa 6254 1890}%
\special{pa 6280 1906}%
\special{pa 6310 1920}%
\special{pa 6330 1930}%
\special{sp}%
%
\special{pn 8}%
\special{pa 6060 1250}%
\special{pa 6060 1330}%
\special{fp}%
\special{sh 1}%
\special{pa 6060 1330}%
\special{pa 6080 1264}%
\special{pa 6060 1278}%
\special{pa 6040 1264}%
\special{pa 6060 1330}%
\special{fp}%
\end{picture}%
\]
\caption{generators $x_1, \ldots ,x_{2g}$ of $\pi_1(\Sigma_{g,1},*)$ and a simple closed curve $\zeta$}\label{fig1}
\end{figure}
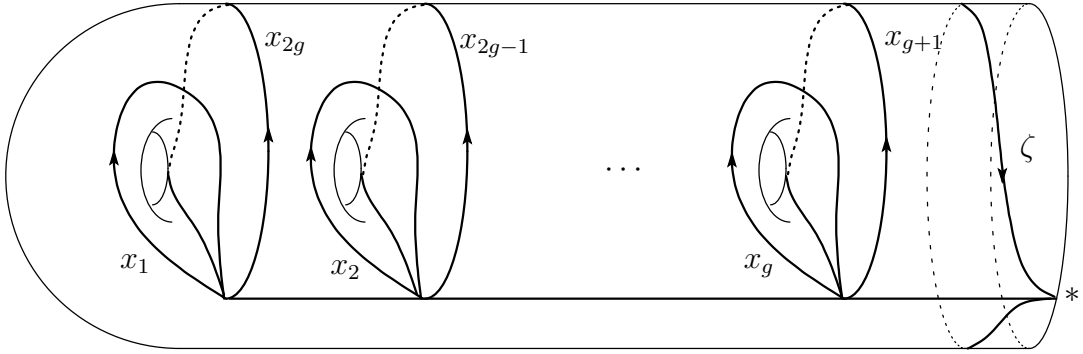

Then the homology classes $e_1, \ldots, e_{2g}$ of $x_1, \ldots, x_{2g}$ form a
symplectic basis of the homology group $H_1(\Sigma_{g,1},\Z)$. Using this symplectic basis, we can identify
$\mathrm{Aut}(H_1(\Sigma_{g,1},\Z))$ as the general linear group $\GL(2g,\Z)$. Under this identification, the image of $\mu_{M}$ is considered as
the symplectic group
\[ \Sp(2g,\Z):=\{X \in \GL(2g,\Z) \ | \ {}^{t}X J X = J \} \,\,\, \mathrm{for} \,\,\, J= \left(
\begin{array}{cc}
0 & I_g \\
-I_g & 0
\end{array}
\right) \]
where $I_g$ is the identity matrix of degree $g$.

\vspace{0.5em}

Next, we consider an embedding of the mapping class group $\M_{g,1}$ into the automorphism group of a free group of rank $2g$.
For $n \geq 2$ let $F_n$ be a free group of rank $n$ with basis $x_1, \ldots, x_n$.
We denote by $\mathrm{Aut}\,F_n$ the automorphism group of $F_n$. Let $H$ be the abelianization $H_1(F_n,\Z)$ of $F_n$
and $\mu : \mathrm{Aut}\,F_n \rightarrow \mathrm{Aut}\,H$ a natural homomorphism induced from
the abelianization map $F_n \rightarrow H$. Throughout the paper, we identify $\mathrm{Aut}\,H$ with the general linear group $\GL(n,\Z)$ by
fixing a basis $e_1, \ldots, e_n$ of $H$ induced from the basis $x_1, \ldots , x_n$ of $F_n$.
By a classical work of Nielsen \cite{Ni1}, a finite presentation of $\mathrm{Aut}\,F_n$ is obtained.
Observing the images of the generators of Nielsen's presentation, we see that $\rho$ is surjective.
The kernel $\mathrm{IA}_n$ of $\rho$ is called the IA-automorphism group of $F_n$. The IA-automorphism group $\mathrm{IA}_n$
is a free group analogue of the Torelli group $\mathcal{I}_{g,1}$.

\vspace{0.5em}

Now, throughout the paper, we identify $\pi_1(\Sigma_{g,1}, *)$ with $F_{2g}$, and $H_1(\Sigma_{g,1},\Z)$ with $H$ for $n=2g$ using the basis above.
Then the action of $\M_{g,1}$ on $\pi_1(\Sigma_{g,1}, *)=F_{2g}$ induces a natural homomorphism
\[ \varphi : \M_{g,1} \rightarrow \mathrm{Aut}\,F_{2g}. \]
By a classical work due to Dehn and Nielsen, it is known that $\varphi$ is injective. More precisely, we have
\begin{thm}[Dehn and Nielsen]
For any $g \geq 1$, we have
\[ \varphi({\M_{g,1}}) = \{ \sigma \in \mathrm{Aut}\,F_{2g} \,\, | \,\, \zeta^{\sigma} = \zeta \} \]
where $\zeta=[x_1,x_{2g}] [x_2,x_{2g-1}] \cdots [x_{g},x_{g+1}] \in F_{2g}$, namely $\zeta$ is a homotopy class of a simple closed curve on
$\Sigma_{g,1}$ parallel to the boundary.
\end{thm}
For $n=2g$, we have $\mu_{\M} = \mu \circ \varphi : \M_{g,1} \rightarrow \Sp(2g,\Z)$, and
a commutative diagram:
\[
\xymatrix{
  1\ar[r] & \mathrm{IA}_{2g}\ar[r] & \mathrm{Aut}\,F_{2g}  \ar[r]^{\mu} & \GL(2g,\Z)\ar[r] &  1 \\
  1 \ar[r] & \mathcal{I}_{g,1} \ar@{^{(}->}[u]_{\varphi|_{\mathcal{I}_{g,1}}}\ar[r] & \M_{g,1}\ar@{^{(}->}[u]_{\varphi}\ar[r]^{\mu_{\M}}& \Sp(2g,\Z)\ar[r]\ar@{^{(}->}[u] & 1 \\
}
\]

\subsection{Free Lie algebras}

In this subsection, we recall the free Lie algebra generated by $H$, and its derivation algebra.
(See \cite{Ser} and \cite{Re} for basic material concerning the free Lie algebra for instance.)

\vspace{0.5em}

Let $\Gamma_n(1) \supset \Gamma_n(2) \supset \cdots$ be the lower central series of a free group $F_n$ defined by the rule
\[ \Gamma_n(1):= F_n, \hspace{1em} \Gamma_n(k) := [\Gamma_n(k-1),F_n], \hspace{1em} k \geq 2. \]
We denote by $\mathcal{L}_n(k) := \Gamma_n(k)/\Gamma_n(k+1)$ the $k$-th graded quotient of the lower central series of $F_n$,
and by $\mathcal{L}_n := {\bigoplus}_{k \geq 1} \mathcal{L}_n(k)$ the associated graded sum.
The degree $1$ part $\mathcal{L}_n(1)$ of $\mathcal{L}_n$ is just $H$.
Classically, Magnus showed that each of $\mathcal{L}_n(k)$ is a free abelian, and Witt \cite{Wit} gave its rank as follows.
\begin{equation}\label{ex-witt}
 \mathrm{rank}_{\Z}(\mathcal{L}_n(k))=\frac{1}{k} \sum_{d | k} \Mob(d) n^{\frac{k}{d}}
\end{equation}
where $\Mob$ is the M$\ddot{\mathrm{o}}$bius function.
For any $k$, $l \geq 1$, let us consider a bilinear alternating map
\[ [ \, , \,]_{\mathrm{Lie}} : \mathcal{L}_n(k) \times \mathcal{L}_n(l) \rightarrow \mathcal{L}_n(k+l) \]
defined by $[\, [\alpha], [\beta]\, ]_{\mathrm{Lie}} := [\, [\alpha, \beta] \,]$ for any $[\alpha] \in \mathcal{L}_n(k)$ and $[\beta] \in \mathcal{L}_n(l)$,
where $[\alpha, \beta]$ is a commutator in $F_n$, and $[\, [\alpha, \beta] \,]$ is a coset class of $[\alpha, \beta]$ in $\mathcal{L}_n(k+l)$.
Then $[ \, , \,]_{\mathrm{Lie}}$ induces a graded Lie algebra structure of the graded sum $\mathcal{L}_n$.
By a classical work of Magnus, the Lie algebra $\mathcal{L}_n$ is isomorphic to the free Lie algebra generated by $H$.

\vspace{0.5em}

The Lie algebra $\mathcal{L}_n$ is considered as a Lie subalgebra of the tensor algebra generated by $H$ as follows.
Let
\[ T(H):= \Z \oplus H \oplus H^{\otimes 2} \oplus \cdots \]
be the tensor algebra of $H$ over $\Z$. Then $T(H)$ is the
universal enveloping algebra of the free Lie algebra $\mathcal{L}_n$, and the natural map
$\iota : \mathcal{L}_n \rightarrow T(H)$ defined by
\[ [X,Y] \mapsto X \otimes Y - Y \otimes X \]
for $X$, $Y \in \mathcal{L}_n$ 
is an injective graded Lie algebra homomorphism.
We denote by $\iota_k$ the homomorphism of degree $k$ part of $\iota$, and
consider $\mathcal{L}_n(k)$ as a submodule $H^{\otimes k}$ through $\iota_k$.

\vspace{0.5em}

Here, we recall the derivation algebra of the free Lie algebra.
Let $\mathrm{Der}(\mathcal{L}_n)$ be the graded Lie algebra of derivations of $\mathcal{L}_n$.
Namely,
\[ \mathrm{Der}(\mathcal{L}_n) := \{ f : \mathcal{L}_n \xrightarrow{\Z-\mathrm{linear}} \mathcal{L}_n \,|\, f([a,b]) = [f(a),b]+ [a,f(b)], \,\,\,
   a, b \in \mathcal{L}_n \}. \]
For $k \geq 0$, the degree $k$ part of $\mathrm{Der}(\mathcal{L}_n)$ is defined to be
\[ \mathrm{Der}(\mathcal{L}_n)(k) := \{ f \in \mathrm{Der}(\mathcal{L}_n) \,|\, f(a) \in \mathcal{L}_n(k+1), \,\,\, a \in H \}. \]
Then, we have
\[ \mathrm{Der}(\mathcal{L}_n) = \bigoplus_{k \geq 0} \mathrm{Der}(\mathcal{L}_n)(k), \]
and can consider $\mathrm{Der}(\mathcal{L}_n)(k)$ as
\[ \mathrm{Hom}_{\Z}(H,\mathcal{L}_n(k+1)) = H^* {\otimes}_{\Z} \mathcal{L}_n(k+1) \]
for each $k \geq 1$ by the universality of the free Lie algebra.
Let $\mathrm{Der}^+(\mathcal{L}_n)$ be a graded Lie subalgebra of $\mathrm{Der}(\mathcal{L}_n)(k)$ with positive degree.
(See Section 8 of Chapter II in \cite{Bou}.)

\subsection{(Higher) Johnson homomorphisms}\label{Ss-John}

First we recall the Johnson filtration and the Johnson homomorphisms of the automorphism group of a free group.
Then we consider those of the mapping class group.

\vspace{0.5em}

For each $k \geq 1$, let $N_{n,k}:=F_n/\Gamma_n(k+1)$ of $F_n$ be the free nilpotent group of class $k$ and rank $n$, and $\mathrm{Aut}\,N_{n,k}$
its automorphism group.
Since the subgroup $\Gamma_n(k+1)$ is characteristic in $F_n$, the group $\mathrm{Aut}\,F_n$ naturally acts on $N_{n,k}$ from the right.
This action induces a homomorphism $\mathrm{Aut}\,F_n \rightarrow \mathrm{Aut}\,N_{n,k}$.
Let $\mathcal{A}_n(k)$ be the kernel of this homomorphism. Then the groups $\mathcal{A}_n(k)$ define a descending filtration
\[ \mathrm{IA}_n = \mathcal{A}_n(1) \supset \mathcal{A}_n(2) \supset \cdots \]
This filtration is called the Johnson filtration of $\mathrm{Aut}\,F_n$.
Set $\mathrm{gr}^k (\mathcal{A}_n) := \mathcal{A}_n(k)/\mathcal{A}_n(k+1)$.
Andreadakis \cite{And} originally studied the Johnson filtration, and obtained basic and important properties of it as follows:
\begin{thm}[Andreadakis, \cite{And}]\label{T-And} \quad 
\begin{enumerate}[$(i)$]
\item For any $k$, $l \geq 1$, $\sigma \in \mathcal{A}_n(k)$ and $x \in \Gamma_n(l)$, $x^{-1} x^{\sigma} \in \Gamma_n(k+l)$.
\item For any $k$, $l \geq 1$, $[\mathcal{A}_n(k), \mathcal{A}_n(l)] \subset \mathcal{A}_n(k+l)$.
In other words, the Johnson filtration is a descending central filtration of $\mathrm{IA}_n$.
\item For any $k \geq 1$, $\mathrm{gr}^k (\mathcal{A}_n)$ is a free abelian group of finite rank.
\end{enumerate}
\end{thm}

\vspace{0.5em}

In order to study the structure of ${\mathrm{gr}}^k (\mathcal{A}_n)$,
the $k$-th Johnson homomorphism of $\mathrm{Aut}\,F_n$ is defined as follows.
\begin{dfn}
For each $k \geq 1$, define a homomorphism
$\tilde{\tau}_k : \mathcal{A}_n(k) \rightarrow \mathrm{Hom}_{\Z}(H, {\mathcal{L}}_n(k+1))$ by
\[ \sigma \hspace{0.3em} \mapsto \hspace{0.3em} \big{(} x \,\, \mathrm{mod} \,\, \Gamma_n(2) \mapsto x^{-1} x^{\sigma}
     \,\, \mathrm{mod} \,\, \Gamma_n(k+2) \big{)}, \hspace{1em} x \in F_n. \]
Then the kernel of $\tilde{\tau}_k$ is just $\mathcal{A}_n(k+1)$. 
Hence it induces an injective homomorphism
\[ \tau_k : \mathrm{gr}^k (\mathcal{A}_n) \hookrightarrow \mathrm{Hom}_{\Z}(H, \mathcal{L}_n(k+1))
       = H^* \otimes_{\Z} \mathcal{L}_n(k+1). \]
This homomorphism is called the $k$-th Johnson homomorphism of $\mathrm{Aut}\,F_n$.
\end{dfn}

Here we consider actions of $\mathrm{GL}(n,\Z)=\mathrm{Aut}\,F_n/\mathrm{IA}_n$.
First, since each term of the lower central series of $F_n$ is a characteristic subgroup, $\mathrm{Aut}\,F_n$ naturally acts on it,
and hence each of the graded quotient $\mathcal{L}_n(k)$.
By (i) of Theorem {\rmfamily \ref{T-And}}, we see that the action of $\mathrm{IA}_n$ on $\mathcal{L}_n(k)$ is trivial.
Thus the action of $\mathrm{GL}(n,\Z) = \mathrm{Aut}\,F_n/\mathrm{IA}_n$ on $\mathcal{L}_n(k)$ is well-defined. \\
On the other hand, since each term of the Johnson filtration is a normal subgroup of $\mathrm{Aut}\,F_n$,
the group $\mathrm{Aut}\,F_n$ naturally acts on $\mathcal{A}_n(k)$ by conjugation, and hence each of the graded quotient
$\mathrm{gr}^k (\mathcal{A}_n)$. By (ii) of Theorem {\rmfamily \ref{T-And}}, we see that the action of $\mathrm{IA}_n$ on
$\mathrm{gr}^k (\mathcal{A}_n)$ is trivial. Namely, we may consider $\mathrm{gr}^k (\mathcal{A}_n)$ as a $\mathrm{GL}(n,\Z) = \mathrm{Aut}\,F_n/\mathrm{IA}_n$-module.
With respect to the actions above, we see that
The Johnson homomorphism $\tau_k$ is $\mathrm{GL}(n,\Z)$-equivariant for each $k \geq 1$. 

\vspace{0.5em}

Furthermore, we remark that the sum of the Johnson homomorphisms forms a Lie algebra homomorphism as follows.
Let ${\mathrm{gr}}(\mathcal{A}_n) := \bigoplus_{k \geq 1} {\mathrm{gr}}^k (\mathcal{A}_n)$ be the graded sum of
$\mathrm{gr}^k (\mathcal{A}_n)$.
The graded sum ${\mathrm{gr}}(\mathcal{A}_n)$ has a graded Lie algebra structure induced from the commutator bracket on $\mathrm{IA}_n$
by an argument similar to that of the free Lie algebra $\mathcal{L}_n$.
Then the sum of the Johnson homomorphisms
\[ \tau := \bigoplus_{k \geq 1} \tau_k : {\mathrm{gr}}(\mathcal{A}_n) \rightarrow \mathrm{Der}^{+}(\mathcal{L}_n) \]
is a graded Lie algebra homomorphism. (See also Theorem 4.8 in \cite{Mo1}.)

\vspace{0.5em}

In the following, we consider three central subfiltration of the Johnson filtration of $\mathrm{Aut}\,F_n$, and
\lq\lq restrictions" of the Johnson homomorphism $\tau_k$.

\vspace{0.5em}

The first one is the lower central series of $\mathrm{IA}_n$.
Let $\mathcal{A}_n'(k)$ be the lower central series of $\mathrm{IA}_n$ with $\mathcal{A}_n'(1)=\mathrm{IA}_n$.
Since the Johnson filtration is central, $\mathcal{A}_n'(k) \subset \mathcal{A}_n(k)$ for each $k \geq 1$.
Set $\mathrm{gr}^k(\mathcal{A}_n') := \mathcal{A}_n'(k)/\mathcal{A}_n'(k+1)$.
Then $\mathrm{GL}(n,\Z)$ naturally acts on each of $\mathrm{gr}^k(\mathcal{A}_n')$, and
the restriction of $\tilde{\tau}_k$ to $\mathcal{A}_n'(k)$ induces a $\mathrm{GL}(n,\Z)$-equivariant homomorphism
\[ \tau_k' : \mathrm{gr}^k (\mathcal{A}_n') \rightarrow H^* \otimes_{\Z} \mathcal{L}_n(k+1). \]
We also call $\tau_k'$ the Johnson homomorphism of $\mathrm{Aut}\,F_n$.
We remark that if we denote by $i_k : \mathrm{gr}^k(\mathcal{A}_n') \rightarrow \mathrm{gr}^k(\mathcal{A}_n)$ the homomorphism induced from the
inclusion $\mathcal{A}_n'(k) \hookrightarrow \mathcal{A}_n(k)$, then $\tau_k' = \tau_k \circ i_k$ for each $k \geq 1$.
Similarly to the sum $\tau$ of $\tau_k$s, the sum $\tau' := \oplus_{k \geq 1} \tau_k' : {\mathrm{gr}}(\mathcal{A}_n') \rightarrow \mathrm{Der}^{+}(\mathcal{L}_n)$
is a graded Lie algebra homomorphism.

\vspace{0.5em}

Let $\mathcal{C}_n(k)$ be a quotient module of $H^{\otimes k}$ by the action of cyclic group $\Cyc_k$ of order $k$ on the components: 
\[ \mathcal{C}_n(k) = H^{\otimes k} \big{/} \langle a_1 \otimes a_2 \otimes \cdots \otimes a_k - a_2 \otimes a_3 \otimes \cdots \otimes a_k \otimes a_1
   \,|\, a_i \in H \rangle. \]
In \cite{S11}, we determined the cokernel of the rational Johnson homomorphisms $\tau_k'$ in stable range. Namely, we have
\begin{thm}[Satoh, \cite{S11}]\label{T-S11}
For any $k \geq 2$ and $n \geq k+2$,
\[ \mathrm{Coker}(\tau_{k, \Q}') \cong \mathcal{C}_n^{\Q}(k). \]
\end{thm}
We also remark that in our previous paper \cite{ES}, we studied the $\GL$-irreducible decomposition of $\mathcal{C}_n^{\Q}(k)$. For more details, see Proposition \ref{prop:ES} and Proposition \ref{prop:mult}.

\vspace{0.5em}

Next, we consider the Johnson filtration of the mapping class group. By Dehn and Nielsen's classical work, we can consider $\M_{g,1}$ as a subgroup
of $\mathrm{Aut}\,F_{2g}$ as above. Under this embedding, set $\M_{g,1}(k) := \M_{g,1} \cap \A_{2g}(k)$ for each $k \geq 1$.
Then we have a descending filtration
\[ \mathcal{I}_{g,1} = \M_{g,1}(1) \supset \M_{g,1}(2) \supset \cdots \]
of the Torelli group $\mathcal{I}_{g,1}$. This filtration is called the Johnson filtration of $\M_{g,1}$. 
Set $\mathrm{gr}^k (\mathcal{M}_{g,1}) := \mathcal{M}_{g,1}(k)/\mathcal{M}_{g,1}(k+1)$.
For each $k \geq 1$, the mapping class group $\M_{g,1}$ acts on $\mathrm{gr}^k (\mathcal{M}_{g,1})$ by conjugation.
This action induces that of $\mathrm{Sp}(2g,\Z)=\M_{g,1}/\mathcal{I}_{g,1}$ on it.

\vspace{0.5em}

By an argument similar to that of $\mathrm{Aut}\,F_n$, the Johnson homomorphisms of $\M_{g,1}$ are defined as follows.
For $n=2g$ and $k \geq 1$, consider the restriction of
$\tilde{\tau}_k : \mathcal{A}_{2g}(k) \rightarrow \mathrm{Hom}_{\Z}(H,{\mathcal{L}}_{2g}(k+1))$ to $\M_{g,1}(k)$.
Then its kernel is just $\mathcal{M}_{g,1}(k+1)$. Hence we obtain an injective homomorphism
\[ \tau_k^{\M} : \mathrm{gr}^k (\mathcal{M}_{g,1}) \hookrightarrow \mathrm{Hom}_{\Z}(H, \mathcal{L}_{2g}(k+1))
       = H^* \otimes_{\Z} \mathcal{L}_{2g}(k+1). \]
The homomorphism $\tau_k^{\M}$ is $\mathrm{Sp}(2g,\Z)$-equivariant, and is called the $k$-th Johnson homomorphism of $\M_{g,1}$.
If we consider a $\GL(2g,\Z)$-module $H$ as a $\Sp(2g,\Z)$-module, then $H^* \cong H$ by the Poincar\'{e} duality.
Hence, in the following, we canonically identify the target $H^* \otimes_{\Z} \mathcal{L}_{2g}(k+1)$ of $\tau_k^{\M}$ with
$H \otimes_{\Z} \mathcal{L}_{2g}(k+1)$.

\vspace{0.5em}

Historically, the Johnson filtration of $\mathrm{Aut}\,F_n$ was originally studied by Andreadakis \cite{And} in 1960's as mentioned above.
On the other hand, the Johnson filtration and the Johnson homomorphisms of $\M_{g,1}$ were begun to study by D. Johnson \cite{Jo1} in 1980's
who determined the abelianization of the Torelli subgroup of the mapping class group of a surface in \cite{Jo4}. In particular, he showed that
$\mathrm{Im}(\tau_1^{\M}) \cong \Lambda^3 H$ as an $\Sp(2g,\Z)$-module, and it gives the free part of $H_1(\mathcal{I}_{g,1},\Z)$.

\vspace{0.5em}

Now, let us recall the fact that the image of $\tau_k^{\M}$ is contained in a certain $\Sp(2g,\Z)$-submodule of $H \otimes_{\Z} \mathcal{L}_{2g}(k+1)$,
due to Morita \cite{Mo1}.
In general, for any $n \geq 1$,
let $H \otimes_{\Z} \mathcal{L}_{n}(k+1) \rightarrow \mathcal{L}_{n}(k+2)$ be a $\GL(n,\Z)$-equivariant homomorphism defined by
\[ a \otimes X \mapsto [a,X], \hspace{1em} \mathrm{for} \hspace{1em} a \in H, \,\,\, X \in \mathcal{L}_{n}(k+1). \]
For $n=2g$, we denote by $\h_{g,1}(k)$ the kernel of this homomorphism:
\[ \h_{g,1}(k) := \mathrm{Ker}(H \otimes_{\Z} \mathcal{L}_{2g}(k+1) \rightarrow \mathcal{L}_{2g}(k+2)). \]
Then Morita \cite{Mo1} showed that the image $\mathrm{Im}(\tau_k^{\M})$ is contained in $\h_{g,1}(k)$.
Therefore, to determine how different is $\mathrm{Im}(\tau_k^{\M})$ from $\h_{g,1}(k)$ is one of the most basic problems. 
Throughout the paper, the cokernel $\mathrm{Coker}(\tau_k^{\M})$ of $\tau_k^{\M}$ always means the quotient $\Sp(2g,\Z)$-module $\h_{g,1}(k)/\mathrm{Im}(\tau_k^{\M})$.
So far, the $\Sp$-module structure of $\mathrm{Coker}(\tau_{k,\Q}^{\M})$ is determined for $1 \leq k \leq 4$ as follows.

{\small
\begin{center}
{\renewcommand{\arraystretch}{1.3}
\begin{tabular}{|c|l|l|l|} \hline
  $k$ & $\mathrm{Im}(\tau_{k,\Q}^{\M})$                     & $\mathrm{Coker}(\tau_{k,\Q}^{\M})$   &                         \\ \hline
  $1$ & $[1^3] \oplus [1]$                                  & $0$                                  & Johnson \cite{Jo1}      \\ \hline
  $2$ & $[2^2] \oplus [1^2] \oplus [0]$                     & $0$                                  & Morita \cite{Mo0}, Hain \cite{Hai} \\ \hline
  $3$ & $[3,1^2] \oplus [2,1]$                              & $[3]$                                & Asada-Nakamura \cite{AN}, Hain \cite{Hai}  \\ \hline
  $4$ & $[4,2] \oplus [3,1^3] \oplus [2^3] \oplus 2 [3,1] \oplus [2,1^2] \oplus 2[2]$  & $[2,1^2] \oplus [2]$    
                                                                                 & Morita \cite{Mo}  \\ \hline
\end{tabular}}
\end{center}
}

Morita \cite{Mo1} showed that the symmetric tensor product $S^k H_{\Q}$ appears in the $\Sp$-irreducible decomposition of $\mathrm{Coker}(\tau_{k,\Q}^{\M})$
for odd $k \geq 3$ using the Morita trace map. In general, however, to determine the cokernel of $\tau_k^{\M}$ is a difficult problem.

\vspace{0.5em}

Here, we recall a remarkable result of Hain.
As an $\Sp(2g,\Z)$-module, we consider $\h_{g,1}(k)$ as a submodule of the degree $k$ part $\mathrm{Der}(\mathcal{L}_n)(k)$ of the derivation algebra
of $\mathcal{L}_n$. On the other hand, the graded sum
\[ \h_{g,1} := \bigoplus_{k \geq 1} \h_{g,1}(k) \]
naturally has a Lie subalgebra structure of $\mathrm{Der}^+(\mathcal{L}_n)$. Therefore we obtain a graded Lie algebra homomorphism
\[ \tau^{\M} := \bigoplus_{k \geq 1} \tau_k^{\M} : {\mathrm{gr}}(\mathcal{M}_{g,1}) \rightarrow \h_{g,1}. \]
Then we have
\begin{thm}[Hain \cite{Hai}]\label{T-Hain}
The Lie subalgebra $\mathrm{Im}(\tau_{\Q}^{\M})$ is generated by the degree one part
$\mathrm{Im}(\tau_{1,\Q}^{\M}) = \Lambda^3 H_{\Q}$ as a Lie algebra.
\end{thm}

\vspace{0.5em}

Finally, we consider the lower central series of the Torelli group, and reformulate Hain's result above.
Let $\M_{g,1}'(k)$ be the lower central series of $\mathcal{I}_{g,1}$, and set
$\mathrm{gr}^k(\M_{g,1}') := \M_{g,1}'(k)/\M_{g,1}'(k+1)$ for $k \geq 1$.
Let ${\tau'_{k}}^{\M}
 : \mathrm{gr}^k(\M_{g,1}') \rightarrow H \otimes_{\Z} \mathcal{L}_{2g}(k+1)$ be an $\Sp$-equivariant homomorphism
induced from the restriction of $\tilde{\tau}_k$ to $\M_{g,1}'(k)$. Then we have
\begin{prop}[Hain, \cite{Hai}]\label{prop:Im}
 \ We have $\mathrm{Im}(\tau_{k,\Q}^{\M})=\mathrm{Im}(\mtau)$ for each $k \geq 1$. 
\end{prop}
For $n=2g$, we have the following commutative diagram:
\[
\xymatrix{
    & \Im\tau_{k,\Q}' \ar@{^{(}->}[rr]   &  
    &   H_{\Q}^* \otimes_{\Q} \mathcal{L}_{2g}^{\Q}(k+1) \ar@{->>}[r] & H_{\Q}^{\otimes{k}} \ar@{->>}[r] & \mathcal{C}_{2g}^{\Q}(k) \\
  \Im \tau_{k,\Q}^{\M} \ar@{=}[r] & \Im \mtau
 \ar@{^{(}->}[u] \ar@{^{(}->}[r] 
    & \h_{g,1}^{\Q}(k) \ar@{^{(}->}[r] & H_{\Q} \otimes_\Q \mathcal{L}_{2g}^{\Q}(k+1) \ar@{->>}[rr]\ar[u]_{\wr} & & \mathcal{L}_{2g}^{\Q}(k+1)
}
\]

\section{Highest weight theory for $\Sp(2g,\Q)$}
\subsection{Irreducible highest weight modules for $\Sp(2g,\Q)$}
\quad Let us consider the general linear group $\GL(n,\Q)$ and the symplectic group
\[ \Sp(2g,\Q):=\{X \in \GL(2g,\Q) \ | \ {}^{t}XJX=J\} \,\,\, \mathrm{for} \,\,\, J=\left(
\begin{array}{cc}
0 & I_g \\
-I_g & 0
\end{array}
\right) \]
where $I_g$ is the identity matrix of degree $g$. 
We fix a maximal torus
\[ T_n=\{\diag(x_1, \ldots ,x_n) \ | \ x_j \neq 0, \ 1 \le j \le n \} \] of $\GL(n,\Q)$.
The intersection $\Sp(2g,\Q) \cap T_{2g}=\{\diag(x_1, \ldots ,x_n,x_n^{-1}, \ldots ,x_1^{-1})\}$ gives a maximal torus of $\Sp(2g,\Q)$.
We also fix this maximal torus and write $T_{2g}^{Sp}$. \\
\quad We define one-dimensional representations $\varepsilon_i$ of $T_n$ by $\varepsilon_i(\diag(x_1, \ldots ,x_n))=x_i$. Then 
\begin{eqnarray*}
P_{\GL(n,\Q)}&:=&\{\lambda_1\varepsilon_1+ \cdots +\lambda_n\varepsilon_n \ | \ \lambda_i \in \bb{Z}, \ 1 \le i \le n \}\cong \bb{Z}^n, \\
P^+_{\GL(n,\Q)}&:=&\{\lambda_1\varepsilon_1+ \cdots +\lambda_n\varepsilon_n \in P_{\GL(n,\Q)}\ | \ \lambda_1 \ge \lambda_2 \ge \cdots \ge \lambda_n\}
\end{eqnarray*}
give the weight lattice and the set of dominant integral weights of $\GL(n,\Q)$ respectively.
If $n=2g$, we can restrict $\varepsilon_i$ to $T_{2g}^{Sp}$ for $1 \le i \le g$. Then
\begin{eqnarray*}
P_{\Sp(2g,\Q)}&:=&\{\lambda_1\varepsilon_1+ \cdots +\lambda_g\varepsilon_g \ | \ \lambda_i \in \bb{Z}, \ 1 \le i \le g \}\cong \bb{Z}^g, \\
P^+_{\Sp(2g,\Q)}&:=&\{\lambda_1\varepsilon_1+ \cdots +\lambda_g\varepsilon_g \in P_{\Sp(2g,\Q)}\ | \ \lambda_1 \ge \lambda_2 \ge \cdots \ge \lambda_g \ge 0\}
\end{eqnarray*}
give the weight lattice and the set of dominant integral weights of $\Sp(2g,\Q)$ respectively.
In particular, there exists a bijection between $P^+_{\Sp(2g,\Q)}$ and the set of partitions such that $\ell(\lambda) \le g$. \\
\quad Let $G$ be a classical group $\GL(n,\Q)$ or $\Sp(2g,\Q)$, $T$ its fixed maximal torus, $P$ its weight lattice and
$P^+$ the set of dominant integral weight with respect to $T$.
For a rational representation $V$ of $G$, there exists an irreducible decomposition $V=\bigoplus_{\lambda \in P}V_\lambda$ as a $T$-module
where $V_\lambda:=\{v \in V \ | \ tv=\lambda(t)v \ \text{for any} \ t \in T\}$. We call this decomposition a weight decomposition of $V$ with respect to $T$.
If $V_\lambda \neq \{0\}$, then we call $\lambda$ a weight of $V$. For a weight $\lambda$,
a non-zero vector $v \in V_{\lambda}$ is call a weight vector of weight $\lambda$. \\
\quad Let $U$ be the subgroup of $G$ consists of all upper unitriangular matrices in $G$. For a rational representation $V$ of $G$,
we define $V^U:=\{v \in V \ | \ uv=v \ \text{for all} \ u \in U\}$. We call a non-zero vector $v \in V^U$ a maximal vector of $V$. This subspace $V^U$ is $T$-stable. Thus, as a $T$-module, $V^U$ has a irreducible decomposition $V^U=\bigoplus_{\lambda \in P}V^U_\lambda$ where $V^U_\lambda:=V^U \cap V_\lambda$.
\begin{thm}[Cartan-Weyl's highest weight theory] \quad 
\begin{enumerate}[$(i)$]
\item Any rational representation of $V$ is completely reducible.
\item Suppose $V$ is an irreducible rational representation of $G$. Then $V^U$ is one-dimensional, and the weight $\lambda$ of $V^U=V_\lambda^U$ belongs to $P^+$.
We call this $\lambda$ the highest weight of $V$, and any non-zero vector $v \in V^U_\lambda$ is called a highest weight vector of $V$. 
\item For any $\lambda \in P^+$, there exists a unique (up to isomorphism) irreducible rational representation $L^\lambda$ of $G$
with highest weight $\lambda$. Moreover, for two $\lambda, \mu \in P^+$, $L^\lambda \cong L^\mu$ if and only if $\lambda=\mu$.
\item The set of isomorphism classes of irreducible rational representations of $G$ is parametrized by the set $P^+$ of dominant integral weights.
\item Let $V$ be a rational representation of $G$ and $\chi_V$ a character of $V$ as a $T$-module.
Then for two rational representation $V$ and $W$, they are isomorphic as $G$-modules if and only if $\chi_V=\chi_W$.
\end{enumerate}
\end{thm}

\begin{rem}
We can parametrize the set of isomorphism classes of irreducible rational representations of $\GL(n,\Q)$ by $P^+_{\GL(n,\Q)}$.
On the other hand, we define the determinant representation by $\det^e:\mathrm{GL}(n,\Q) \ni X \to \det{X}^e \in \Q ^\times$.
The highest weight of this representation is given by $(e,e, \cdots ,e) \in P_{\GL(n,\Q)}^+$.
If $\lambda \in P^+$ satisfies $\lambda_n<0$,
then $L^\lambda \cong \det^{-\lambda_n} \otimes L^{(\lambda_1-\lambda_n,\lambda_2-\lambda_n, \ldots ,0)}$.
Moreover the set of isomorphism classes of polynomial irreducible representations is parametrized by the set of partitions $\lambda$
such that $\ell(\lambda) \le n$.
We denote the polynomial representations corresponding to a partition $\lambda$ by $L_{\GL}^{\lambda}$, $L^{(\lambda)}$ or simply $(\lambda)$. 
\end{rem}

\begin{rem}
We can parametrize the set of isomorphism classes of irreducible rational representations of $\Sp(2g,\Q)$
by $P^+_{\Sp(2g,\Q)}\cong\{\lambda_1 \ge \lambda_2 \ge \cdots \ge \lambda_g \ge 0 \ | \ \lambda_i \in \bb{Z}, 1 \le i \le n\}$,
namely the set of partitions $\lambda$ such that $\ell(\lambda) \le g$.
In this paper, we denote the irreducible representation corresponding to $\lambda$ by $L_{\Sp}^{\lambda}$, $L^{[\lambda]}$ or simply $[\lambda]$. \\
\quad Note that the natural representation $H_{\Q}=\Q^{2g}$ of $\Sp(2g,\Q)$ is irreducible with highest weight $(1,0, \ldots ,0)$ and $H_{\Q}^* \cong H_{\Q}$
by the Poincar\'{e} duality. More precisely, we set $i':=2g-i+1$ for each integer $1 \le i \le 2g$. Then for the standard basis $\{e_i\}_{i=1}^{2g}$ of $H_{\Q}$, we see
\begin{eqnarray}
\langle e_i,e_j\rangle=0=\langle e_{i'},e_{j'}\rangle,\quad 
\langle e_i,e_{j'}\rangle=\delta_{ij}=-\langle e_{j'},e_{i}\rangle, \quad (1 \le i \le g). \label{innerp}
\end{eqnarray}
There is an isomorphism $H_\Q \to H^*_\Q$ as $\Sp(2g,\Q)$-modules given by
\begin{eqnarray}
H_\Q \ni v \mapsto \langle \bullet, v\rangle \in H^*_\Q. \label{isomH}
\end{eqnarray}
In general, all irreducible rational representation $[\lambda]$ is isomorphic to its dual. 
\end{rem}
Let us recall Pieri's formula, the simplest version of the decomposition of tensor product representations.
For two partition $\lambda$ and $\mu$ satisfying $\lambda \supset \mu$, the skew shape $\lambda \backslash \mu$ is a vertical strip
if there is at most one box in each row.
\begin{thm}[Pieri's formula]\label{Pieri}
Let $\mu$ be a partition such that $\ell(\mu) \le n$.
Then
\[ L_{\GL}^{(1^k)} \otimes L_{\GL}^{\mu} \cong \bigoplus_{\lambda}L_{\GL}^\lambda, \]
where $\lambda$ runs over the set of partitions obtained by adding a vertical
$k$-strip to $\mu$ such that $\ell(\lambda) \le n$.
\end{thm}

\subsection{Branching rules from $\GL(2g,\Q)$ to $\Sp(2g,\Q)$}
We regard $\Sp(2g,\Q)$ as a subgroup of $\GL(2g,\Q)$.
We consider the restriction of an irreducible polynomial representation $L^{\lambda}_{\GL}$ to $\Sp(2g,\Q)$.
We can give its irreducible decomposition using the Littlewood-Richardson coefficients $\LR_{\lambda\mu}^{\nu}$ as follows.
\begin{thm}[{\cite[25.39]{FH},\cite[Proposition 2.5.1]{KT}}]\label{thm:brspgl}
Let $\lambda=(\lambda_1 \ge \lambda_2 \ge \cdots \ge \lambda_{g}\ge 0)$ be a partition such that $\ell(\lambda) \le g$. Then we have
\[
\Res_{\Sp(2g,\Q)}^{\GL(2g,\Q)}(L_{\GL}^{\lambda}) \cong \bigoplus_{\bar{\lambda}}N_{\lambda\bar{\lambda}}L^{\bar{\lambda}}_{Sp}
\]
where $\bar{\lambda}$ runs over all partitions such that $\ell(\bar{\lambda}) \le g$. Here 
\[
N_{\lambda\bar{\lambda}}=\sum_{\eta}\LR_{\eta\bar{\lambda}}^{\lambda}
\]
where $\eta$ runs over all partitions $\eta=(\eta_1=\eta_2 \ge \eta_3=\eta_4 \ge \cdots)$ with each part occurring an even number of times,
namely $\eta'$ even. Here $\eta'$ is a conjugate partition of $\eta$.
\end{thm}

\begin{rem}\label{rem:LR}
We give a combinatorial description of the Littlewood-Richardson coefficients. (e.g. \cite{FH}, \cite{Mac}.)
For two Young diagrams $\lambda$ and $\mu$ satisfying $\lambda \subset \mu$,
we denote by $\lambda \backslash \mu$ a skew Young diagram, which is the difference of $\lambda$ and $\mu$. 
For a skew Young diagram $\lambda \backslash \mu$ of size $m$, a semistandard tableau of shape $\lambda \backslash \mu$ is an array $T$
of positive integers $1,2, \ldots ,m$ of shape $\lambda \backslash \mu$ that is weakly increasing in every row and strictly increasing in every column.
\begin{enumerate}[(i)]
\item For two partitions $\lambda \supset \mu$, a \textit{semi-standard tableau} on $\lambda \backslash \mu$ is a numbering on $\lambda \backslash \mu \to \bb{Z}_{\ge 1}$ such that the numbers inserted in $\lambda \backslash \mu$ must increase strictly down each column and weakly from left to right along each row. For a semistandard tableau on $\lambda \backslash \mu$, we denote the number of $i$ appearing in this semistandard tableau by $m_i$. We call $(m_1,m_2, \ldots)$ a \textit{weight} of the semistandard tableau. 
\item For a semistandard tableau $T$ on $\lambda \backslash \mu$, we define a sequence $w(T)$ of integers by reading the numbers inserted in $\lambda \backslash \mu$ from right to left in successive rows, starting with top row.
\item For a sequence $w=(a_1a_2 \cdots )$, we denote the number of $i$ appearing in a subsequence $(a_1a_2 \cdots a_r)$ by $m_i(a_1a_2 \cdots a_r)$. A sequence $w$ is a \textit{lattice permutation} if $m_1(a_1a_2 \cdots a_r) \ge m_2(a_1a_2 \cdots a_r) \ge  \cdots $ for any $r\ge 1$. 
\end{enumerate}
The Littlewood-Richardson coefficients $\LR_{\mu\nu}^\lambda$ is the number of semi-standard tableaux $T$ on $\lambda \backslash \mu$ with weight $\nu$ such that $w(T)$ is a lattice permutation. 
\end{rem}

\subsection{Review on the classical Schur-Weyl duality}
\quad For the natural representation $H_{\Q} \cong L^{(1)}$ of $\mathrm{GL}(n,\Q)$, we consider the $k$-th tensor product representation
$\rho_k:\mathrm{GL}(n,\Q) \to \mathrm{GL}(H_{\Q}^{\otimes{k}})$ of $H_{\Q}$.
For each $k \geq 1$,
the symmetric group $\mf{S}_{k}$ of degree $k$ naturally acts on
the space $H_\Q^{\otimes{k}}$ from the right as a permutation of the components.
Since these two actions are commutative, we can decompose $H_{\Q}^{\otimes{k}}$ as a $(\mathrm{GL}(n,\Q) \times \mf{S}_k)$-module.
Let us recall this irreducible decomposition, called the Schur-Weyl duality for $\mathrm{GL}(n,\Q)$ and $\mf{S}_k$.
\begin{thm}[Schur-Weyl's duality for $\mathrm{GL}(n,\Q)$ and $\mf{S}_k$]\label{thm:SWgl}\quad 
\begin{enumerate}[(i)]
\item Let $\lambda$ be a partition of $k$ such that $\ell(\lambda) \le n$. 
There exists a non-zero maximal vector $v_\lambda$ with weight $\lambda$ satisfying the following three conditions:
\begin{enumerate}[(a)]
\item The $\mf{S}_k$-invariant subspace $S^\lambda:=\sum_{\sigma \in \mf{S}_k}\Q{v_\lambda}\cdot \sigma$ gives an irreducible representation of $\mf{S}_k$.
\item The subspace $(H_{\Q}^{\otimes{k}})^U_\lambda$ of weight $\lambda$ coincides with the subspace $S^\lambda$, where $U$ is the fixed unipotent subgroup of $\GL(n,\Q)$
      consisting of upper unitriangular matrices.
\item The $\mathrm{GL}(n,\Q)$-module generated by $v_\lambda$ is isomorphic to the irreducible representation $L^{(\lambda)}_{\GL}$ of $\mathrm{GL}(n,\Q)$
      with highest weight $\lambda$.
\end{enumerate}
\item We have the irreducible decomposition:
\[
H_{\Q}^{\otimes{k}} \cong \bigoplus_{\lambda=(\lambda_1 \ge \cdots \ge \lambda_n\ge 0) \vdash k}L^\lambda \boxtimes S^\lambda.
\]
as $(\mathrm{GL}(n,\Q) \times \mf{S}_k)$-modules.
\item Suppose $n \ge k$. Then $\{S^\lambda \ | \ \lambda \vdash k\}$ gives a complete representatives of irreducible representations of $\mf{S}_k$.
\end{enumerate}
\end{thm}
\begin{rem}\label{rem:SW} \quad 
\begin{enumerate}[(i)] 
\item The irreducible representation $S^\lambda$ of $\mf{S}_k$ is isomorphic to the following $\mf{S}_k$-module. \\
\quad For a partition $\lambda$ of $k$, we define two special Young subgroups $C_\lambda:=\mf{S}_{\lambda_1} \times \mf{S}_{\lambda_2} \times \cdots $
and $R_\lambda:=\mf{S}_{\lambda'_1} \times \mf{S}_{\lambda'_2} \times \cdots$ of $\mf{S}_k$. Here a partition $\lambda'=(\lambda'_1,\lambda'_2, \ldots)$
is the conjugate partition of $\lambda$. In the group algebras of these two groups, we find idempotents
\[
a_\lambda=\dfrac{1}{|R_\lambda|}\sum_{\sigma \in R_\lambda}\sigma\in \Q R_\lambda, \,\,\, \mathrm{and} \,\,\,
b_\lambda=\dfrac{1}{|C_\lambda|}\sum_{\sigma \in C_\lambda}\sgn(\sigma)\sigma\in \Q C_\lambda.
\]
Then $c_\lambda=|R_\lambda||C_\lambda|a_\lambda b_\lambda$ gives an idempotent in $\Q\mf{S}_k$, called the Young symmetrizer for $\lambda$.
The right ideal $c_\lambda \cdot \Q\mf{S}_k$ in $\Q\mf{S}_k$ gives an irreducible $\mf{S}_k$-module which is isomorphic to $S^\lambda$ above. 
\item We construct $v_\lambda$ appearing in the theorem above by the following way. \\
\quad First, we define $v_1 \wedge v_2 \wedge \cdots \wedge v_r$ to be an anti-symmetrizer
\[ \sum_{\sigma \in \mf{S}_r}\sgn(\sigma) (v_{1} \otimes v_{2} \otimes \cdots \otimes v_{r}) \cdot \sigma \in H_{\Q}^{\otimes{r}}. \]
For the natural base $\{e_i\}_{i=1}^n$ of $H_{\Q}$, we define 
\begin{eqnarray}
v_{\lambda}:=(e_1 \wedge \cdots \wedge e_{\lambda'_1}) \otimes (e_1 \wedge \cdots \wedge e_{\lambda'_2}) \otimes \cdots \in H_{\Q}^{\otimes{k}}.  \label{maxgl}
\end{eqnarray}
Note that $v_\lambda$ is a maximal vector of weight $\lambda$ and
\[ v_{\lambda}=(e_1 \otimes \cdots \otimes e_{\lambda'_1} \otimes e_1 \otimes \cdots \otimes e_{\lambda'_2} \otimes \cdots )\cdot c_\lambda. \]
This $v_\lambda$ gives our desirable vector in the theorem above. 
\end{enumerate}
\end{rem}

\subsection{Brauer-Schur-Weyl's duality}

The first two subsection is based on \cite{HY} and \cite{Hu}. The last one is based on \cite{Ra}.
\subsubsection{Brauer algebras}
\quad Let us define the Brauer algebra $B_k(-2g)$ with a parameter $-2g$ and size $k$. 
\begin{dfn}
The Brauer algebra $B_k(-2g)$ over $\Q$ is a unital associative $\Q$-algebra with the following generators and defining relations:
\begin{eqnarray*}
generators&:&s_1, \ldots ,s_{k-1}, \gm_1, \ldots ,\gm_{n-1}, \\
relations&:&s_i^2=1, \quad \gm_i^2=(-2g) \gm_i, \quad \gm_is_i=\gm_i=s_i\gm_i, \quad (1 \le i \le k-1), \\
&{}& s_is_j=s_js_i, \quad s_i\gm_j=\gm_js_i, \quad \gm_i\gm_j=\gm_j\gm_i, \quad (1 \le i<j-1 \le k-2), \\
&{}& s_is_{i+1}s_i=s_{i+1}s_is_{i+1}, \quad \gm_i\gm_{i+1}\gm_i=\gm_i, \quad \gm_{i+1}\gm_i\gm_{i+1}=\gm_{i+1}, \quad (1 \le i \le k-2), \\
&{}& s_i\gm_{i+1}\gm_i=s_{i+1}\gm_i, \quad \gm_{i+1}\gm_is_{i+1}=\gm_{i+1}s_i, \quad (1 \le i \le k-2).
\end{eqnarray*}
\end{dfn}
\begin{rem}
The Brauer algebra $B_k(-2g)$ is obtained by the following diagrammatic way. \\
\quad First of all, the Brauer $k$ diagram is a diagram with specific $2k$ vertices arranged in two rows of $k$ each,
the top rows and the bottom rows, and exactly $k$ edges such that every vertex is joined to another vertex (distinct from itself) by exactly one edge.
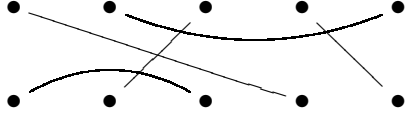
\[
\xymatrix{
\bullet\ar@{-}[rrrd] & \bullet\ar@/_1pc/@{-}[rrr] & \bullet\ar@{-}[ld] & \bullet\ar@{-}[dr] & \bullet \\
\bullet\ar@/^1pc/@{-}[rr] & \bullet & \bullet & \bullet & \bullet 
}
\]
We define a multiplication of two diagrams as follows. We compose two diagrams $D_1$ and $D_2$ by identifying the bottom row of $D_1$ with the top row of $D_2$
such that the $i$-th vertex in the bottom row of $D_1$ is coincided with the $i$-th vertex in the top row of $D_2$.
The result is a graph, with a certain number, $n(D_1,D_2)$, of interior loops. After removing the interior loops and the identified vertices,
retaining the edges and remaining vertices, we obtain a new Brauer $k$-diagram $D_1 \circ D_2$.
Then we define a multiplication $D_1 \cdot D_2$ by $(-2g)^{n(D_1,D_2)}D_1 \circ D_2$. 
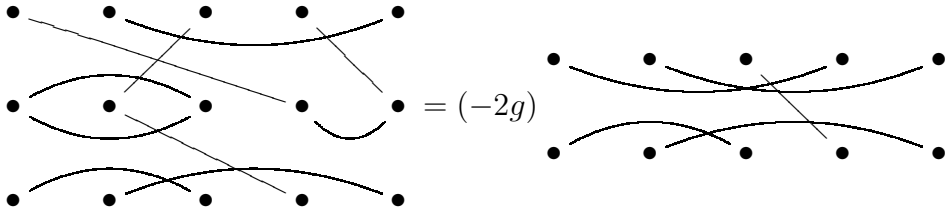
\[
\begin{minipage}[h]{90mm}
\xymatrix{
\bullet\ar@{-}[rrrd] & \bullet\ar@/_1pc/@{-}[rrr] & \bullet\ar@{-}[ld] & \bullet\ar@{-}[dr] & \bullet \\
\bullet\ar@/^1pc/@{-}[rr]\ar@/_1pc/@{-}[rr] & \bullet\ar@{-}[rrd] & \bullet & \bullet\ar@/_1pc/@{-}[r] & \bullet \\
\bullet\ar@/^1pc/@{-}[rr] & \bullet\ar@/^1pc/@{-}[rrr] & \bullet & \bullet & \bullet \\
}
\end{minipage}
=(-2g)
\begin{minipage}[h]{90mm}
\xymatrix{
\bullet\ar@/_1pc/@{-}[rrr] & \bullet\ar@/_1pc/@{-}[rrr] & \bullet\ar@{-}[rd] & \bullet & \bullet \\
\bullet\ar@/^1pc/@{-}[rr] & \bullet\ar@/^1pc/@{-}[rrr] & \bullet & \bullet & \bullet 
}
\end{minipage}
\]
The Brauer algebra $B_k(-2g)$ is defined as $\bb{Q}$-linear space with a basis being the set of the Brauer $k$-diagrams and
the multiplication of two elements given by the linear extension of a product above. \\
\quad The generators $s_i$ and $\gm_i$ correspond to the following diagrams.
\begin{eqnarray*}
\gm_i&=&\begin{minipage}[h]{90mm}
\xymatrix{
\stackrel{1}{\bullet}\ar@{-}[d] & \cdots & \stackrel{i}{\bullet}\ar@<-1.5mm>@{-}[r] & \stackrel{i+1}{\bullet} & \cdots & \stackrel{k}{\bullet}\ar@{-}[d] \\
\bullet & \cdots & \bullet\ar@{-}[r] & \bullet & \cdots & \bullet 
}
\end{minipage} \quad (1 \le i \le k-1)\\
s_i&=&\begin{minipage}[h]{90mm}
\xymatrix{
\stackrel{1}{\bullet}\ar@{-}[d] & \cdots & \stackrel{i}{\bullet}\ar@{-}[rd] & \stackrel{i+1}{\bullet}\ar@{-}[ld] & \cdots & \stackrel{k}{\bullet}\ar@{-}[d] \\
\bullet & \cdots & \bullet & \bullet & \cdots & \bullet 
}
\end{minipage}
\quad  (1 \le i \le k-1)
\end{eqnarray*}
\end{rem}

\subsubsection{Decomposition of tensor spaces (Brauer-Schur-Weyl's duality)}

\quad Let us recall the inner product on $H_\Q$ defined by (\ref{innerp}). 
Set $i':=2g-i+1$ for each integer $1 \le i \le 2g$. For the standard basis $\{e_i\}_{i=1}^{2g}$ of $H_{\Q}$, we see
\begin{eqnarray*}
\langle e_i,e_j\rangle=0=\langle e_{i'},e_{j'}\rangle,\quad 
\langle e_i,e_{j'}\rangle=\delta_{ij}=-\langle e_{j'},e_{i}\rangle, \quad (1 \le i \le g). 
\end{eqnarray*}
For each integer $1 \le i \le 2g$, we define 
\begin{eqnarray}
e_i^*=\left\{
\begin{array}{ll}
e_{i'}, & (1 \le i \le g), \\
-e_{i'}, & (g+1 \le i \le 2g).
\end{array}
\right.\label{dual}
\end{eqnarray}
Then both of $\{e_i\}_{i=1}^{2g}$ and $\{e_i^*\}_{i=1}^{2g}$ are basis for $H_{\Q}$ such that one is dual to the other in the sense that
$\langle e_i,e_j^* \rangle=\delta_{ij}$ for any $i,j$. \\
\quad The following lemma is obvious, but important to generalize the Schur-Weyl duality for $\Sp(2g,\Q)$.
\begin{lem}\label{lem:inv}
An element
\[ \omega := \sum_{i=1}^{2g}e_i \otimes e_i^* \in H_{\Q}^{\otimes{2}} \]
is invariant under the action of $\Sp(2g,\Q)$ on $H_{\Q}^{\otimes{2}}$.
\end{lem}
We define a right action of $B_k(-2g)$ on $H_{\Q}^{\otimes{k}}$ as follows.
\begin{prop}
There is a right action of $B_k(-2g)$ on $H_{\Q}^{\otimes{k}}$ which is defined on generators by
\begin{eqnarray*}
(v_{i_1} \otimes \cdots \otimes v_{i_k}) \cdot \gamma_j&:=&-v_{i_1} \otimes \cdots \otimes v_{i_{j-1}} \otimes \left(
\sum_{r=1}^{2g}e_k \otimes e_k^*
\right) \otimes v_{i_{j+2}} \otimes \cdots \otimes v_{i_k}, \\
(v_{i_1} \otimes \cdots \otimes v_{i_k}) \cdot s_j&:=&-v_{i_1} \otimes \cdots \otimes v_{i_{j-1}} \otimes v_{i_{j+1}}
   \otimes v_{i_j} \otimes v_{i_{j+2}} \otimes \cdots \otimes v_{i_k},
\end{eqnarray*}
for any $v_{i_1}, \ldots ,v_{i_k} \in H_{\Q}$.
Moreover, this action commutes with that of $\Sp(2g,\Q)$.
\end{prop}
Here we state the Brauer-Schur-Weyl duality.
\begin{thm}[Brauer-Schur-Weyl's duality for $\Sp(2g,\Q)$ and $B_k(-2g)$]\label{thm:SWsp}\quad 
\begin{enumerate}[(i)]
\item Let $\lambda$ be a partition of $k-2j$ for $0 \le j \le \lfloor\frac{k}{2}\rfloor$ such that $\ell(\lambda) \le g$.
Then there exists a maximal vector $v_{\lambda} \in H_{\Q}^{\otimes{k}}$ with highest weight $\lambda$ satisfying the following three conditions:
\begin{enumerate}[(a)]
\item A $B_k(-2g)$-submodule
\[ D^{\lambda}:=\sum_{\sigma \in B_k(-2g)}\Q v_{\lambda}\cdot \sigma \]
of $H_{\Q}^{\otimes{k}}$ gives an irreducible representation of $B_k(-2g)$.
\item The subspace $(H_{\Q}^{\otimes{k}})^U_\lambda$ of $H_{\Q}^{\otimes{k}}$ coincides with $D^{\lambda}$.
Here $U$ is the fixed unipotent subgroup for $\Sp(2g,\Q)$.
\item The $\Sp(2g,\Q)$-module generated by $v_\lambda$ is isomorphic to the irreducible representation $L_{Sp}^{[\lambda]}$
of $Sp(2g,\Q)$ with highest weight $\lambda$.
\end{enumerate}
\item We have the irreducible decomposition 
\[
H_{\Q}^{\otimes{k}} \cong \bigoplus_{j=0}^{\lfloor\frac{k}{2}\rfloor}\bigoplus_{\lambda \vdash k-2j,\ell(\lambda) \le g}L_{Sp}^{[\lambda]}\boxtimes D^\lambda.
\]
as an $(\Sp(2g,\Q) \times B_k(-2g))$-module. 
\item Suppose $g \ge k$. Then $\{D^{\lambda} \ | \ \lambda \vdash k-2j \ (0 \le j \le \lfloor\frac{k}{2}\rfloor)\}$
gives a complete representatives of irreducible representations of $B_k(-2g)$.
\end{enumerate}
\end{thm}
In our purpose of this paper, to observe an explicit construction of $v_\lambda$ and a description of $D^\lambda$ is important. 
\begin{thm}[{\cite[Definition 3.9, Lemma 3.10, Lemma 4.8]{Hu}}]\label{thm:max}\quad 
\begin{enumerate}[(i)]
\item For a partition $\lambda$ of $k-2j$ for $0 \le j \le \lfloor\frac{k}{2}\rfloor$ such that $\ell(\lambda) \le g$,
a maximal vector $v_\lambda$ is given by
\[
v_\lambda:=\omega^{\otimes{j}} \otimes (e_1 \wedge \cdots \wedge e_{\lambda'_1} ) \otimes (e_1 \wedge \cdots \wedge e_{\lambda'_2}) \otimes \cdots. 
\]
\item We regard a subalgebra generated by $s_i \ (1 \le i \le k-1)$ in $B_k(-2g)$ as a group algebra $\Q\mf{S}_k$.
Then the right module $v_\lambda \cdot B_k(-2g)$ coincides with $v_\lambda \cdot \Q\mf{S}_k$ as a $\Q$-vector space.
\end{enumerate}
\end{thm}
\subsubsection{Character values and decompositions of $D^\lambda$ as an $\mf{S}_k$-module}
\quad We give a branching low of the irreducible $B_k(-2g)$-modules $D^\lambda$ as $\mf{S}_k$-modules.
But confusingly, the algebra $\Q\mf{S}_k$ has an involution $\iota:\sigma \mapsto \sgn(\sigma) \sigma$,
and the action of a subalgebra generated by $s_i$'s in $B_k(-2g)$ on $H_{\Q}^{\otimes k}$ is twisted by this involution.
Therefore a $\Q\mf{S}_k$-module $D$ is isomorphic to $\mathrm{\mb{sgn}} \otimes D$ as an $\iota(\Q\mf{S}_k)$-module.
Here $\mathrm{\mb{sgn}}$ is the signature representation of $\mf{S}_k$. Note that an irreducible $\mf{S}_k$-module $S^{\nu'}$
is isomorphic to $\mathrm{\mb{sgn}} \otimes S^\nu$. \\
\quad In our purpose, we consider the ordinary (untwisted) action of $\mf{S}_k$ on $H_{\Q}^{\otimes{k}}$ in the following theorem (ii).
\begin{thm}[{\cite[Theorem 5.1]{Ra}}]\quad 
\begin{enumerate}[(i)]
\item For a partition $\lambda$ of $k-2j$ for $0 \le j \le \lfloor\frac{k}{2}\rfloor$ such that $\ell(\lambda) \le g$,
let $\chi^\lambda_{B_k(-2g)}$ be the irreducible character of $D^\lambda$. Then we have
\[
\chi^{\lambda}_{B_k(-2g)}(\sigma)=\sum_{\nu \vdash k,\nu \supset \lambda'}\left(\sum_{\beta\text{:even}}\LR_{\lambda'\beta}^{\nu}\right)\chi^\nu_{\mf{S}_k}(\sigma).
\]
for any $\sigma \in \mf{S}_k \subset$ (a subalgebra generated by $\{s_i\}_{i=1}^{k-1}$).
Here $\chi^\nu_{\mf{S}_k}$ is an irreducible character of $\mf{S}_k$ associated to a partition $\nu$ of $k$.  The number $\LR$ is the Littlewood-Richardson coefficient. The even partition $\beta=(\beta_1,\beta_2, \ldots )$ is a partition such that any parts $\beta_i$ are even.
\item We have the irreducible decomposition of $D^\lambda$ is given by 
\[
\bigoplus_{\nu \vdash k,\nu \supset \lambda'}
(S^{\nu'})^{\oplus \sum_{\beta{\text{:even}}}\LR_{\lambda'\beta}^\nu}
\]
with respect to the ordinary $\mf{S}_k$-action on $H_{\Q}^{\otimes{k}}$.
\end{enumerate}
\end{thm}
\begin{rem}
For a partition $\lambda \vdash k-2j$, we have the following dimension formula:
\[
\dim{D^\lambda}={}_kC_{2j}(2j-1)!! \cdot \dim{S^\lambda}.
\]
This gives the multiplicity of $L_{\Sp}^{\lambda}$ in $H_{\Q}^{\otimes{k}}$. 
For $\lambda=0$, the formula above is nothing but \cite[Lemma 4.1]{Mo}.
\end{rem}

\section{Dynkin-Specht-Weyman's idempotent and the free Lie algebras}
\quad Let us consider the right action of $\mf{S}_{k+2}$ on $H_{\Q}^{\otimes(k+2)}$. Set $\sigma_i:=s_{i-1}s_{i-2} \cdots s_1$ for each $2 \le i \le k+2$, and 
\[
\theta_{k+2}:=(1-\sigma_2) \cdots (1-\sigma_{k+2}) \in \Q\mf{S}_{k+2}.
\]
This element characterizes the degree $(k+2)$-nd part $\mathcal{L}_{2g}^{\Q}(k+2)$ of the free Lie algebra $\mathcal{L}_{2g}^{\Q}$ generated by $H_{\Q}=\Q^{2g}$
as follows. (e.g., {\cite[Theorem 2.1]{Ga}, \cite[Theorem 8.16]{Re}, \cite[Lemma 4.5]{Mo}}.)
\begin{thm}[Dynkin-Specht-Wever]\label{thm:DSW} \quad 
\begin{enumerate}[(i)]
\item $\theta_{k+2}^2=(k+2)\theta_{k+2}$. We call an element $\frac{1}{k+2}{\theta_{k+2}}$ the Dynkin-Specht-Wever idempotent.
\item For $v_1 \otimes v_2 \otimes \cdots \otimes v_{k+2} \in H_{\Q}^{\otimes{k+2}}$,
a left-normed element $[v_1, v_2, \ldots , v_{k+2}] \in \mathcal{L}_{2g}^{\Q}(k+1)$ coincides with
$(v_1 \otimes v_2 \otimes \cdots \otimes v_{k+2}) \cdot \theta_{k+2}$.
Hence the right action of $\theta_{k+2}$ on $H_{\Q}^{\otimes {k+2}}$ induces a projection $H_{\Q}^{\otimes {k+2}} \to \mathcal{L}_{2g}^{\Q}(k+1)$,
and $H_{\Q}^{\otimes{k+2}} \cdot \theta_{k+2}$ is isomorphic to $\mathcal{L}_{2g}^{\Q}(k+2)$.
\item For $v \in H_{\Q}^{\otimes(k+2)}$, the following two conditions are equivalent;
\begin{enumerate}[(a)]
\item $v \in \mathcal{L}_{2g}^{\Q}(k+2)$,
\item $v \cdot \theta_{k+2}=(k+2)v$.
\end{enumerate}
\end{enumerate}
\end{thm}
Recall that we need to consider the $\Sp(2g,\Q)$-module
\[
\h_{g,1}^{\Q}(k)=\Ker(H_\Q \otimes_\Q \mathcal{L}_{2g}^{\Q}(k+1) \to \mathcal{L}_{2g}^{\Q}(k+2)).
\]
To characterize $\mf{h}_{g,1}^{\Q}(k)$ in $H_{\Q}^{\otimes{k+2}}$, let us consider a subgroup $P$ of $\mf{S}_{k+2}$
which fixes $1$. Namely, $P$ is isomorphic to $\mf{S}_{k+1}$. Set 
\[
\theta_P:=(1-s_2)(1-s_3s_2) \cdots (1-s_{k+1}s_k \cdots s_2).
\]
We can regard this element in $\Q{P}$ as the Dynkin-Specht-Wever idempotent for $P$. Using this element,
we obtain a characterization of $\h_{g,1}^{\Q}(k)$ as the following theorem.
\begin{prop}[{\cite[Proposition 4.6]{Mo}}]\label{prop:cri}
For $v \in H_{\Q}^{\otimes(k+2)}$, the following two conditions are equivalent;
\begin{enumerate}[(i)]
\item $v \in \mf{h}_{g,1}^{\Q}(k)$,
\item $v \cdot \theta_{P}=(k+1)v$ and $v \cdot \sigma_{k+2}=v$.
\end{enumerate}
\end{prop}

\begin{cor}\label{cor:cri} 
We have 
\[
\theta_P \cdot (1+\sigma_{k+2}+\sigma_{k+2}^2+ \cdots +\sigma_{k+2}^{k+1}) \cdot \theta_P=(k+1)\theta_P \cdot (1+\sigma_{k+2}+\sigma_{k+2}^2+ \cdots +\sigma_{k+2}^{k+1})
\]
on $H_{\Q}^{\otimes{k+2}}$.
Thus we obtain
\[ v \cdot \theta_P(1+\sigma_{k+2}+\sigma_{k+2}^2+ \cdots +\sigma_{k+2}^{k+1}) \in \mf{h}_{g,1}^{\Q}(k) \]
for any $v \in H_{\Q}^{\otimes{k+2}}$.
\end{cor}
\begin{proof}
Let us recall the following expansions of a left-normed element in the free Lie algebra:
\begin{eqnarray}
[x_1, x_2, \ldots , x_m]=\sum(-1)^rx_{i_1} \otimes \cdots \otimes x_{i_r} \otimes x_1 \otimes x_{j_1} \otimes \cdots \otimes x_{j_{m-r-1}} \label{expl}
\end{eqnarray}
where the sum runs over all integers $r$ and tuples $(i_1, \ldots ,i_r)$ and $(j_1, \ldots ,j_{m-r-1})$ of integers satisfying the conditions
\[
0 \le r \le m-1,  \quad m \ge i_1>\cdots >i_r\ge 2, \quad 2 \le j_1< \cdots <j_{m-r-1} \le m.
\]
(See e.g., \cite[Lemma 1.1]{Re}.)
The expansion above is equivalent to 
\begin{eqnarray}
\sum(-1)^{r-1}x_{i_1} \otimes \cdots \otimes x_{i_r} \otimes x_2 \otimes x_{j_1} \otimes \cdots \otimes x_{j_{m-r-1}} \label{expl2}
\end{eqnarray}
where the sum runs over all integers $r$ and tuples $(i_1, \ldots ,i_r)$ and $(j_1, \ldots ,j_{m-r-1})$ of integers satisfying the conditions
\[
0 \le r \le m-1,  \quad m \ge i_1>\cdots >i_r\ge 1, \quad 1 \le j_1< \cdots <j_{m-r-1} \le m
\]
and $i_1, \ldots ,i_r,j_1, \ldots ,j_{m-r-1} \neq 2$. \\
\quad Note that $(v_1 \otimes \cdots \otimes v_{k+2})\cdot \theta_P=v_1 \otimes [v_2, \ldots ,v_{k+2}]$ for any $v_1, \ldots ,v_{k+2} \in H_{\Q}$.
To prove our statement, we shall prove
\begin{eqnarray}
&{}&(v_1 \otimes \cdots \otimes v_{k+2})\cdot \theta_P\cdot (1+\sigma+ \cdots +\sigma^{k+1}) \nonumber \\
&=& v_1 \otimes [v_2, \ldots ,v_{k+2}]-\sum_{j=2}^{k+2}v_j \otimes [[v_2,v_3, \ldots ,v_{j-1}],[v_{j+1}, [v_{j+2}, \cdots ,[v_{k+2},v_1] \cdots ]]]. \label{claim:exp}
\end{eqnarray}
In the formula above, the righthand side is contained in $H_{\Q} \otimes_{\Q} \mathcal{L}_{2g}^{\Q}(k+1)$. Therefore if (\ref{claim:exp}) is true,
by Theorem \ref{thm:DSW}, we obtain our claim. \\
\quad To prove the formula (\ref{claim:exp}), we set
\[ x_1=[v_1, \ldots ,v_{j-1}], \,\, x_2=v_j, \,\, x_3=v_{j+1}, \, \ldots , \,\, x_{k+4-s}=v_{k+2}. \]
Then applying the formula (\ref{expl2}),
we expand $(v_1 \otimes \cdots \otimes v_{k+2})\cdot \theta_P$ like as 
\[
v_1 \otimes \sum (-1)^{r-1}x_{i_1} \otimes \cdots \otimes x_{i_r} \otimes x_2 \otimes x_{j_1} \otimes \cdots \otimes x_{j_{k+3-s-r}}
\]
satisfying the similar condition for (\ref{expl2}). Hence, in $(v_1 \otimes \cdots \otimes v_{k+2})\cdot \theta_P\cdot (1+\sigma+ \cdots +\sigma^{k+1})$,
the terms which first part is equal to $v_j$ are given by 
\begin{eqnarray}
v_j \otimes \sum (-1)^{r-1}x_{j_1} \otimes \cdots \otimes x_{j_{k+3-s-r}} \otimes v_1 \otimes x_{i_1} \otimes \cdots \otimes x_{i_r} \label{expj}
\end{eqnarray}
satisfying the conditions
\[
0 \le r \le k+3-s, \quad 1 \le j_1< \cdots <j_{k+3-s-r} \le k+2, \quad k+2 \ge i_1> \cdots >i_r \ge 1
\]
and $i_1$, $\ldots$, $i_r$, $j_1$, $\ldots$, $j_{k+3-s-r} \neq 2$. \\
\quad On the other hand, note that the following expansion of a right-normed element in a free Lie algebra:
\[
[x_1,[x_2, \cdots ,[x_{m-1},x_m]\cdots]]=\sum(-1)^rx_{j_1} \otimes \cdots \otimes x_{j_{m-r-1}} \otimes x_m \otimes x_{i_1} \otimes \cdots \otimes x_{i_r},
\]
where the sum runs over all integers $r$, tuples $(i_1, \ldots ,i_r)$ and $(j_1, \ldots ,j_{m-r-1})$ of integers satisfying the conditions
\[
0 \le r \le m-1, \quad m \ge i_1>\cdots >i_r\ge 1, \quad 1 \le j_1< \cdots <j_{m-r-1} \le m.
\]
Applying this formula to (\ref{expj}), we obtain 
\[
-v_j \otimes [x_1,[x_2, \cdots ,[x_{k+4-s},v_1]]]
\]
for $x_1=[v_1, \ldots ,v_{j-1}]$, $x_2=v_j$, $x_3=v_{j+1}$, $\ldots$, $x_{k+4-s}=v_{k+2}$. Thus we have the formula (\ref{claim:exp}).
\end{proof}

\section{Multiplicities in $\Res_{\Cyc_k}^{\mf{S}_k}S^\lambda$ via Kra\'{s}kiewicz-Weyman's combinatorial description}

Let $\Cyc_k$ be a cyclic group of order $k$. Take a generator $\sigma_k$ of $\Cyc_k$ and a primitive $k$-th root $\zeta_k \in \C$ of unity.
In this section, we consider representations of the cyclic group $\Cyc_k$ over an intermediate field $\Q(\zeta_k) \subset \mb{K} \subset \C$. \\
\quad To begin with, we define one-dimensional representations (or characters) $\chi_k^j : \Cyc_k \to \mb{K}^\times$ by $\chi_k^j(\sigma_k)=\zeta_k^j$
for $0 \le j \le k-1$.
Especially, we denote the trivial representation $\chi_k^0$ by $\triv_k$.
The set of isomorphism classes of irreducible representations of $\Cyc_k$ is given by $\{\chi_k^j, \,\, 0 \le j \le k-1 \}$.
Consider $\Cyc_k$ as a subgroup of $\mf{S}_k$ by an embedding $\sigma_k^i \mapsto (1 \, 2 \, \cdots \, k)^i$ for $0 \leq i \leq k-1$.
Let us recall Kra\'{s}kiewicz-Weyman's combinatorial description for the branching rules of irreducible $\mf{S}_k$-modules $S^\lambda$
to the cyclic subgroup $\Cyc_k$. To do this, first we define a major index of a standard tableau.
\begin{dfn}\label{def:mj}
For a standard tableau $T$, we define the descent set of $T$ to be the set of entries $i$ in $T$ such that $i+1$ is located in a lower row than that
which $i$ is located. We denote by $D(T)$ the descent set of $T$.
The major index of $T$ is defined by
\[ \maj(T):=\sum_{i \in D(T)}i. \]
If $D(T)=\phi$, we set $\maj(T)=0$.
\end{dfn}
\begin{thm}[{\cite{KW}, \cite[Theorem 8.8, 8.9]{Re}, \cite[Theorem 8.4]{Ga}}]\label{thm:KW}
The multiplicity of $\chi_k^j$ in $\Res_{\Cyc_k}^{\mf{S}_k}S^\lambda$ is equal to the number of standard tableaux with shape $\lambda$
satisfying $\maj(T) \equiv j$ modulo $k$.
\end{thm}
\begin{exa}\label{exa:KW} For $k \ge 2$, we have the following table on the multiplicities of $\triv_k=\chi_j^0$ and $\chi_j^1$.
\[
{\footnotesize 
\begin{array}{|c|c|c|c|c|} \hline
\lambda & T & \text{major index} & \text{mult. of $\triv_m$} & \text{mult. of $\chi_m^1$} \\
\hline
& & & & \\
(m) & \begin{minipage}[h]{30mm}{\hspace{2.5mm}
\unitlength 0.1in
\begin{picture}(  9.7500,  2.4000)(  3.8500, -6.4000)
%
\special{pn 8}%
\special{pa 400 400}%
\special{pa 640 400}%
\special{pa 640 640}%
\special{pa 400 640}%
\special{pa 400 400}%
\special{fp}%
%
\special{pn 8}%
\special{pa 640 400}%
\special{pa 880 400}%
\special{pa 880 640}%
\special{pa 640 640}%
\special{pa 640 400}%
\special{fp}%
\put(5.2000,-5.2000){\makebox(0,0){$1$}}%
\put(7.6000,-5.2000){\makebox(0,0){$2$}}%
%
\special{pn 8}%
\special{pa 880 400}%
\special{pa 1120 400}%
\special{pa 1120 640}%
\special{pa 880 640}%
\special{pa 880 400}%
\special{fp}%
%
\special{pn 8}%
\special{pa 1120 400}%
\special{pa 1360 400}%
\special{pa 1360 640}%
\special{pa 1120 640}%
\special{pa 1120 400}%
\special{fp}%
\put(12.4000,-5.2000){\makebox(0,0){$m$}}%
\put(10.0000,-5.2000){\makebox(0,0){$\cdots$}}%
\end{picture}%
}\end{minipage}
& 0 & 1 & 0 \\
& & & & \\
\hline
& & & & \\
(m-1,1) & \begin{minipage}[h]{30mm}{\hspace{2.5mm}
\unitlength 0.1in
\begin{picture}(  9.7500,  4.8000)(  3.8500, -8.8000)
%
\special{pn 8}%
\special{pa 400 400}%
\special{pa 640 400}%
\special{pa 640 640}%
\special{pa 400 640}%
\special{pa 400 400}%
\special{fp}%
\put(5.2000,-5.2000){\makebox(0,0){$1$}}%
%
\special{pn 8}%
\special{pa 880 400}%
\special{pa 1120 400}%
\special{pa 1120 640}%
\special{pa 880 640}%
\special{pa 880 400}%
\special{fp}%
%
\special{pn 8}%
\special{pa 1120 400}%
\special{pa 1360 400}%
\special{pa 1360 640}%
\special{pa 1120 640}%
\special{pa 1120 400}%
\special{fp}%
\put(12.4000,-5.2000){\makebox(0,0){$m$}}%
\put(10.0000,-5.2000){\makebox(0,0){$\cdots$}}%
%
\special{pn 8}%
\special{pa 640 400}%
\special{pa 880 400}%
\special{pa 880 640}%
\special{pa 640 640}%
\special{pa 640 400}%
\special{fp}%
\put(7.6000,-5.2000){\makebox(0,0){$2$}}%
%
\special{pn 8}%
\special{pa 400 640}%
\special{pa 640 640}%
\special{pa 640 880}%
\special{pa 400 880}%
\special{pa 400 640}%
\special{fp}%
\put(5.2000,-7.6000){\makebox(0,0){$p$}}%
\end{picture}%
}\end{minipage}
& p-1 & 0 & 1 \\
& (2 \le p \le m) & & & \\
\hline
& & & & \\
(1^m) & \begin{minipage}[h]{30mm}{\hspace{6mm}
\unitlength 0.1in
\begin{picture}(  4.6600,  9.6000)(  1.1000,-13.6000)
%
\special{pn 8}%
\special{pa 336 400}%
\special{pa 576 400}%
\special{pa 576 640}%
\special{pa 336 640}%
\special{pa 336 400}%
\special{fp}%
\put(4.5600,-5.2000){\makebox(0,0){$1$}}%
%
\special{pn 8}%
\special{pa 336 880}%
\special{pa 576 880}%
\special{pa 576 1120}%
\special{pa 336 1120}%
\special{pa 336 880}%
\special{fp}%
%
\special{pn 8}%
\special{pa 336 1120}%
\special{pa 576 1120}%
\special{pa 576 1360}%
\special{pa 336 1360}%
\special{pa 336 1120}%
\special{fp}%
\put(4.5600,-12.4000){\makebox(0,0){$m$}}%
\put(4.7000,-9.7000){\makebox(0,0){$\vdots$}}%
%
\special{pn 8}%
\special{pa 336 640}%
\special{pa 576 640}%
\special{pa 576 880}%
\special{pa 336 880}%
\special{pa 336 640}%
\special{fp}%
\put(4.5600,-7.6000){\makebox(0,0){$2$}}%
\end{picture}%
}\end{minipage}
&
\begin{array}{l}
 \dfrac{m(m-1)}{2} \\
\equiv \left\{
\begin{array}{ll}
0, & \text{$m$ : odd} \\
-\frac{m}{2}, & \text{$m$ : even} 
\end{array}
\right.
\end{array}

& \begin{cases} 1, & \text{$m$ : odd} \\ 0, & \text{$m$ : even}  \end{cases} & \begin{cases} 1, & m=2 \\ 0, & m \neq 2 \end{cases} \\
& & & & \\
\hline
& & & & \\
(2,1^{m-2}) & \begin{minipage}[h]{30mm}{\hspace{4mm}
\unitlength 0.1in
\begin{picture}(  7.1600,  9.6000)(  1.0000,-13.6000)
%
\special{pn 8}%
\special{pa 336 400}%
\special{pa 576 400}%
\special{pa 576 640}%
\special{pa 336 640}%
\special{pa 336 400}%
\special{fp}%
\put(4.5600,-5.2000){\makebox(0,0){$1$}}%
%
\special{pn 8}%
\special{pa 336 880}%
\special{pa 576 880}%
\special{pa 576 1120}%
\special{pa 336 1120}%
\special{pa 336 880}%
\special{fp}%
%
\special{pn 8}%
\special{pa 336 1120}%
\special{pa 576 1120}%
\special{pa 576 1360}%
\special{pa 336 1360}%
\special{pa 336 1120}%
\special{fp}%
\put(4.5600,-12.4000){\makebox(0,0){$m$}}%
\put(4.6000,-9.8000){\makebox(0,0){$\vdots$}}%
%
\special{pn 8}%
\special{pa 336 640}%
\special{pa 576 640}%
\special{pa 576 880}%
\special{pa 336 880}%
\special{pa 336 640}%
\special{fp}%
\put(4.5600,-7.6000){\makebox(0,0){$2$}}%
%
\special{pn 8}%
\special{pa 576 400}%
\special{pa 816 400}%
\special{pa 816 640}%
\special{pa 576 640}%
\special{pa 576 400}%
\special{fp}%
\put(6.9600,-5.2000){\makebox(0,0){$p$}}%
\end{picture}%
}\end{minipage}
& 
\begin{array}{l}
\dfrac{m(m-1)}{2}-(p-1) \\
\equiv \left\{
\begin{array}{ll}
1-p, & \text{$m$ : odd} \\
1-p-\frac{m}{2}, & \text{$m$ : even} 
\end{array}
\right.
\end{array}
& \begin{cases} 1, & \text{$m$ : even} \\ 0, & \text{$m$ : odd} \end{cases} & \begin{cases} 1, & m \neq 2 \\ 0, & m = 2 \end{cases} \\
&  (2 \le p \le m)& & & \\ \hline
\end{array}
}
\]
\end{exa}

\begin{exa}\label{exa:KW2}
For $m \geq 3$ and a partition $\lambda=(m-2,1^2)$, we have 
\begin{enumerate}[(i)]
\item $[\triv_m:\Res_{\Cyc_m}^{\mf{S}_m}S^\lambda]= \begin{cases} (m-2)/2 \hspace{1em} & \mathrm{if} \,\,\, m : \mathrm{even}, \\
                                                                  (m-1)/2 & \mathrm{if} \,\,\, m : \mathrm{odd}. \end{cases}$
\item $[\chi_m^1:\Res_{\Cyc_m}^{\mf{S}_m}S^\lambda]= \begin{cases} (m-3)/2 \hspace{1em} & \mathrm{if} \,\,\, m : \mathrm{odd}, \\
                                                                   (m-2)/2 & \mathrm{if} \,\,\, m : \mathrm{even}. \end{cases}$
\end{enumerate}
In fact, for a partition
\[ T=\begin{minipage}[h]{30mm}
\unitlength 0.1in
\begin{picture}(  9.7500,  7.2000)(  3.8500,-11.2000)
%
\special{pn 8}%
\special{pa 400 400}%
\special{pa 640 400}%
\special{pa 640 640}%
\special{pa 400 640}%
\special{pa 400 400}%
\special{fp}%
\put(5.2000,-5.2000){\makebox(0,0){$1$}}%
%
\special{pn 8}%
\special{pa 880 400}%
\special{pa 1120 400}%
\special{pa 1120 640}%
\special{pa 880 640}%
\special{pa 880 400}%
\special{fp}%
%
\special{pn 8}%
\special{pa 1120 400}%
\special{pa 1360 400}%
\special{pa 1360 640}%
\special{pa 1120 640}%
\special{pa 1120 400}%
\special{fp}%
\put(12.4000,-5.2000){\makebox(0,0){$m$}}%
\put(10.0000,-5.2000){\makebox(0,0){$\cdots$}}%
%
\special{pn 8}%
\special{pa 640 400}%
\special{pa 880 400}%
\special{pa 880 640}%
\special{pa 640 640}%
\special{pa 640 400}%
\special{fp}%
\put(7.6000,-5.2000){\makebox(0,0){$2$}}%
%
\special{pn 8}%
\special{pa 400 640}%
\special{pa 640 640}%
\special{pa 640 880}%
\special{pa 400 880}%
\special{pa 400 640}%
\special{fp}%
\put(5.2000,-7.6000){\makebox(0,0){$p$}}%
%
\special{pn 8}%
\special{pa 400 880}%
\special{pa 640 880}%
\special{pa 640 1120}%
\special{pa 400 1120}%
\special{pa 400 880}%
\special{fp}%
\put(5.2000,-10.0000){\makebox(0,0){$q$}}%
\end{picture}%
\end{minipage}, \]
its major index is given by $\maj(T)=p+q-2$ for $2 \le p<q \le m$. Then $\maj(T) \equiv 0 \pmod{m}$ if and only if $p+q=m+2$.
Hence we have the number of standard tableaux of shape $\lambda$ is equal to $\dfrac{m}{2}-1$ for odd $m$ and $\dfrac{m-1}{2}$ for even $m$.
On the other hand, $\maj(T) \equiv 1 \pmod{m}$ if and only if $p+q=m+3$.
Hence the number of standard tableaux of shape $\lambda$ is equal to $\dfrac{m-3}{2}$ for odd $m$ and $\dfrac{m-2}{2}$ for even $m$. 
\end{exa}
\begin{exa}\label{exa:KW3}
For $m \geq 4$ and a partition $\lambda=(2^2,1^{m-4})$, we have
\[
 \ [\chi_m^1:\Res_{\Cyc_m}^{\mf{S}_m}S^\lambda]=\left\{
\begin{array}{ll}
\frac{m-3}{2} & \text{if $m$ is odd}, \\
\frac{m-4}{2} & \text{if $m \equiv 0 \ (\mathrm{mod} \, 4)$}, \\
\frac{m-2}{2} & \text{if $m \equiv 2 \ (\mathrm{mod} \, 4)$}.
\end{array}
\right.
\] 
To prove this, we consider the following two kind of standard tableaux of shape $\lambda$:
\[
T_{p,q}=\begin{minipage}[h]{20mm}
\unitlength 0.1in
\begin{picture}(  7.0600,  9.6000)(  1.1000,-13.6000)
%
\special{pn 8}%
\special{pa 336 400}%
\special{pa 576 400}%
\special{pa 576 640}%
\special{pa 336 640}%
\special{pa 336 400}%
\special{fp}%
\put(4.5600,-5.2000){\makebox(0,0){$1$}}%
%
\special{pn 8}%
\special{pa 336 880}%
\special{pa 576 880}%
\special{pa 576 1120}%
\special{pa 336 1120}%
\special{pa 336 880}%
\special{fp}%
%
\special{pn 8}%
\special{pa 336 1120}%
\special{pa 576 1120}%
\special{pa 576 1360}%
\special{pa 336 1360}%
\special{pa 336 1120}%
\special{fp}%
\put(4.5600,-12.4000){\makebox(0,0){$m$}}%
%
\special{pn 8}%
\special{pa 336 640}%
\special{pa 576 640}%
\special{pa 576 880}%
\special{pa 336 880}%
\special{pa 336 640}%
\special{fp}%
\put(4.5600,-7.6000){\makebox(0,0){$2$}}%
%
\special{pn 8}%
\special{pa 576 400}%
\special{pa 816 400}%
\special{pa 816 640}%
\special{pa 576 640}%
\special{pa 576 400}%
\special{fp}%
\put(6.9600,-5.2000){\makebox(0,0){$p$}}%
%
\special{pn 8}%
\special{pa 576 640}%
\special{pa 816 640}%
\special{pa 816 880}%
\special{pa 576 880}%
\special{pa 576 640}%
\special{fp}%
\put(6.9600,-7.6000){\makebox(0,0){$q$}}%
\put(4.7000,-9.8000){\makebox(0,0){$\vdots$}}%
\end{picture}%
\end{minipage}
 \ (2 \le p<p+1<q \le m), \quad 
T_p=\begin{minipage}[h]{20mm}
\unitlength 0.1in
\begin{picture}(  7.0600,  9.6000)(  1.1000,-13.6000)
%
\special{pn 8}%
\special{pa 336 400}%
\special{pa 576 400}%
\special{pa 576 640}%
\special{pa 336 640}%
\special{pa 336 400}%
\special{fp}%
\put(4.5600,-5.2000){\makebox(0,0){$1$}}%
%
\special{pn 8}%
\special{pa 336 880}%
\special{pa 576 880}%
\special{pa 576 1120}%
\special{pa 336 1120}%
\special{pa 336 880}%
\special{fp}%
%
\special{pn 8}%
\special{pa 336 1120}%
\special{pa 576 1120}%
\special{pa 576 1360}%
\special{pa 336 1360}%
\special{pa 336 1120}%
\special{fp}%
\put(4.5600,-12.4000){\makebox(0,0){$m$}}%
%
\special{pn 8}%
\special{pa 336 640}%
\special{pa 576 640}%
\special{pa 576 880}%
\special{pa 336 880}%
\special{pa 336 640}%
\special{fp}%
\put(4.5600,-7.6000){\makebox(0,0){$2$}}%
%
\special{pn 8}%
\special{pa 576 400}%
\special{pa 816 400}%
\special{pa 816 640}%
\special{pa 576 640}%
\special{pa 576 400}%
\special{fp}%
\put(6.9600,-5.2000){\makebox(0,0){$p$}}%
%
\special{pn 8}%
\special{pa 576 640}%
\special{pa 816 640}%
\special{pa 816 880}%
\special{pa 576 880}%
\special{pa 576 640}%
\special{fp}%
\put(6.9600,-7.6000){\makebox(0,0){{\tiny$p+1$}}}%
\put(4.7000,-9.8000){\makebox(0,0){$\vdots$}}%
\end{picture}%
\end{minipage}
(3 \le p \le m-1).
\]
Their major indices are given by
\[ \maj(T_{p,q})=\frac{m(m-1)}{2}+2-p-q \hspace{1em} \mathrm{and} \hspace{1em} \maj(T_p)=\frac{m(m-1)}{2}+1-p. \]
If $m$ is odd, $\frac{m(m-1)}{2} \equiv 0 \pmod{m}$.
Thus $\maj(T_{p,q}) \equiv 1 \pmod{m}$ if and only if $p+q=m+1$. The number of such $(p,q)$s is $\frac{m-3}{2}$.
There is no $T_p$ such that $\maj(T_p) \equiv 1 \pmod{m}$.
If $m$ is even, $\frac{m(m-1)}{2} \equiv \frac{m}{2} \pmod{m}$. Since $m \neq 2$, $\maj(T_p) \equiv 1 \pmod{m}$ if and only if $p=\frac{m}{2}$ for $m>4$.
If $m=4$, there is no such $T_p$. \\
\quad On the other hand, $\maj(T_{p,q}) \equiv 1 \pmod{m}$ if and only if $p+q=m+1+\frac{m}{2}$ for $m=4,6,8$ and $p+q=m+1+\frac{m}{2}$,
or $1+\frac{m}{2}$ for $m \ge 10$.
If $m=4$, $6$ or $8$, the number of such $(p,q)$s is $0$, $1$ or $1$ respectively. Suppose $m \ge 10$.
If $m=4M$, $\maj(T_{p,q}) \equiv 1 \pmod{m}$ if and only if $p+q=6M+1$ or $2M+1$. The number of such $(p,q)$s is $(M-1)+(M-2)=2M-3=\frac{m}{2}-3$.
If $m=4M+2$, $\maj(T_{p,q}) \equiv 1 \pmod{m}$ if and only if $p+q=6M+4$ or $2M+2$. The number of such $(p,q)$s is $M+(M-1)=2M-1=\frac{m}{2}-2$.
Therefore we obtain the claim.
\end{exa}

\section{$\Sp$-irreducible components of the Johnson cokernels}
\subsection{Our strategy for detecting $\Sp$-irreducible components}\label{sec:str}

In the rest of this paper, we assume $g \ge k+2$.
To explain our strategy for detecting $\Sp$-irreducible components in the Johnson cokernel of the mapping class group,
let us recall the following diagram as mentioned above:
\[
\xymatrix{
    & \Im\tau_{k,\Q}' \ar@{^{(}->}[rr]   &  
    &   H_{\Q}^* \otimes_{\Q} \mathcal{L}_{2g}^{\Q}(k+1) \ar@{->>}[r] & H_{\Q}^{\otimes{k}} \ar@{->>}[r] & \mathcal{C}_{2g}^{\Q}(k) \\
  \Im \tau_{k,\Q}^{\M} \ar@{=}[r] & \Im \mtau \ar@{^{(}->}[u] \ar@{^{(}->}[r] 
    & \h_{g,1}^{\Q}(k) \ar@{^{(}->}[r] & H_{\Q} \otimes_\Q \mathcal{L}_{2g}^{\Q}(k+1) \ar@{->>}[rr]\ar[u]_{\wr} & & \mathcal{L}_{2g}^{\Q}(k+1)
}
\]
Here we may regard it as a diagram of $\Sp(2g,\Q)$-modules and $\Sp(2g,\Q)$-equivariant homomorphisms.
By Theorem {\rmfamily \ref{T-S11}}, we see $\Coker(\Im{\tau'_{k,\Q}} \hookrightarrow H_{\Q}^* \otimes_\Q \mathcal{L}_{2g}^{\Q}(k+1))$
coincides with $\mathcal{C}_{2g}^{\Q}(k)$ for $2g \geq k+2$.
Observing a natural isomorphism $H^* \otimes_\Q \mathcal{L}_{2g}^{\Q}(k+1) \cong H \otimes_\Q \mathcal{L}_{2g}^{\Q}(k+1)$ induced from the Poincar\'{e} duality,
we obtain $\Sp(2g,\Q)$-equivariant homomorphism $c_k: \mf{h}_{g,1}^{\Q}(k) \to \mathcal{C}_{2g}^{\Q}(k)$.
Note that $\Im \mtau \subset \Im\tau_{k,\Q}'$. Then we have the following criterion for detecting $\Sp$-irreducible components
in the Johnson cokernel $\Coker(\Im \mtau \to \h_{g,1}^{\Q}(k))$.
\begin{prop}\label{prop:coker}
Let $V$ be an irreducible $\Sp(2g,\Q)$-submodule of $\h_{g,1}^{\Q}(k)$. If $c_k(V)$ is a non-trivial (then automatically irreducible) component of
$\mathcal{C}_{2g}^{\Q}(k)$, then $V$ is an irreducible $\Sp(2g,\Q)$-module in $\mathrm{Coker}(\Im \mtau)$.
In particular, if there is a maximal vector $v$ of weight $\lambda$ in $\h_{g,1}^{\Q}(k)$ such that $c_k(v) \neq 0$
(then $c_k(v)$ is a maximal in $\mathcal{C}_{2g}^{\Q}(k)$),
then $v$ gives an $\Sp(2g,\Q)$-irreducible component in $\mathrm{Coker}(\Im \mtau)$
which is isomorphic to the irreducible $\Sp(2g,\Q)$-module $L_{\Sp}^{[\lambda]}$.
\end{prop}
To find such a maximal vector, we use Theorem \ref{thm:max} and Corollary \ref{cor:cri}. Namely, for a maximal vector $v_\lambda$ as in Theorem \ref{thm:max},
we consider $\phi_\lambda:=v_\lambda \cdot \theta_P \cdot (1+\sigma_{k+2}+ \cdots +\sigma_{k+1}^{k+2})$.
If $\phi_\lambda \neq 0$, this is a maximal vector of weight $\lambda$ such that $\phi_\lambda \in \mf{h}_{g,1}^{\Q}(k)$ by Corollary \ref{cor:cri}.
Then we investigate whether $c_k(\phi_\lambda) \in \mathcal{C}_{2g}^{\Q}(k)$ is $0$ or not.

\subsection{Some multiplicity formulae}

In this subsection, we give some explicit multiplicity formulae for $[k]$ and $[1^k]$ in $\mf{h}_{g,1}^{\Q}(k)$ and $\mathcal{C}_{2g}^{\Q}(k)$. 
First, let us recall the multiplicity formulae in our previous paper \cite{ES}. 
\begin{prop}\label{prop:ES}\quad Suppose $n \ge k+2$.
\begin{enumerate}[(i)]
\item For a partition $\lambda$ of $k$, 
\[ [L^{\lambda}_{\GL}: \mathcal{C}_n^{\Q}(k)] = [\triv_k:\Res_{\Cyc_k}^{\mf{S}_k}S^\lambda]. \]
\item For a partition $\lambda$ of $k+2$, 
\[ [L^{\lambda}_{\GL}: \mathcal{L}_n^{\Q}(k+2)] = [\chi_k^1:\Res_{\Cyc_k}^{\mf{S}_k}S^\lambda]. \]
\item For a partition $\lambda$ of $k+2$,
\[ [L_{\GL}^\lambda:H_{\Q} \otimes_\Q \mathcal{L}_n^{\Q}(k+1)] = \sum_{\mu}[L_{\GL}^{\lambda}:\mathcal{L}_n^{\Q}(k+1)] \]
where $\mu$ runs over all partitions obtained by removing a single node.
\end{enumerate}
\end{prop}

\begin{prop}\label{prop:mult}\quad 
\begin{enumerate}[(i)]
\item The multiplicities of the $\Sp(2g,\Q)$-irreducible representation $[k]$ in $\mf{h}_{g,1}^{\Q}(k)$ and $\mathcal{C}_{2g}^{\Q}(k)$ are given by
\[ [L_{\Sp}^{[k]}:\h_{g,1}^{\Q}(k)]= \begin{cases} 1 \hspace{1em} & \mathrm{if} \hspace{1em} k : \mathrm{odd}, \\
                                                   0 & \mathrm{if} \hspace{1em} k : \mathrm{even}, \end{cases} \hspace{1em}
[L_{\Sp}^{[k]}:\mathcal{C}_{2g}^{\Q}(k)]=1.
\]
\item The multiplicities of the $\Sp(2g,\Q)$-irreducible representation $[1^k]$ in $\mf{h}_{g,1}^{\Q}(k)$ and $\mathcal{C}_{2g}^{\Q}(k)$ are given by
\[ [L_{\Sp}^{[1^k]}:\h_{g,1}^{\Q}(k)]=\begin{cases} 1 \hspace{0.5em} & \mathrm{if} \hspace{0.5em} k \equiv 1, 2 \pmod{4}, \\
                                                   0 & \mathrm{if} \hspace{0.5em} \mathrm{otherwise}, \end{cases} \hspace{1em}
[L_{\Sp}^{[1^k]}:\mathcal{C}_{2g}^{\Q}(k)]=\begin{cases} 1 \hspace{0.5em} & \mathrm{if} \hspace{0.5em} k : \mathrm{odd}, \\
                                                   0 & \mathrm{if} \hspace{0.5em} k : \mathrm{even}. \end{cases}
\]
\end{enumerate}
\end{prop}
\begin{proof} \quad We will use irreducible decompositions of the restriction $\Res_{\Sp}^{\GL}$ (See Theorem \ref{thm:brspgl}.)
and Pier's rule (See Theorem \ref{Pieri}.).
\begin{enumerate}[(i)]
\item If $\Res_{\Sp(2g,\Q)}^{\GL(2g,\Q)}L_{\GL}^{(\lambda)}$ has
an $\Sp$-irreducible component $L^{[k]}_{\Sp}$, then a partition $\lambda$ is either $\lambda=(k+1,1)$ or $(k,1^2)$. We have
\begin{eqnarray*}
 \ [L_{\GL}^{(k+1,1)}:H_{\Q} \otimes_\Q \mathcal{L}_{2g}^{\Q}(k+1)]&=&[L_{\GL}^{(k+1)}:\mathcal{L}_{2g}^{\Q}(k+1)]+[L_{\GL}^{(k,1)}:\mathcal{L}_{2g}^{\Q}(k+1)]=1, \\
 \ [L_{\GL}^{(k+1,1)}:\mathcal{L}_{2g}^{\Q}(k+2)]&=&1,
\end{eqnarray*}
\begin{eqnarray*}
 \ [L_{\GL}^{(k,1^2)}:H_{\Q} \otimes_\Q \mathcal{L}_{2g}^{\Q}(k+1)]&=&[L_{\GL}^{(k-1,1^2)}:\mathcal{L}_{2g}^{\Q}(k+1)]+[L_{\GL}^{(k,1)}:\mathcal{L}_{2g}^{\Q}(k+1)], \\
   &=& \begin{cases} \frac{k-2}{2}+1 \hspace{0.5em} & \mathrm{if} \hspace{0.5em} k : \mathrm{even}, \\
                     \frac{k-1}{2}+ 1 & \mathrm{if} \hspace{0.5em} k : \mathrm{odd}, \end{cases} \\
 \ [L_{\GL}^{(k,1^2)}:\mathcal{L}_{2g}^{\Q}(k+2)]
   &=& \begin{cases} \frac{k}{2} \hspace{0.5em} & \mathrm{if} \hspace{0.5em} k : \mathrm{even}, \\
                     \frac{k-1}{2} & \mathrm{if} \hspace{0.5em} k : \mathrm{odd}, \end{cases} \\
 \ [L_{\Sp}^{[k]}:\mathcal{C}_{2g}^{\Q}(k)]&=&[L_{\GL}^{(k)}:\mathcal{C}_{2g}^{\Q}(k)]=1.
\end{eqnarray*}
Thus we obtain the claim. 
\item If $\Res_{\Sp(2g,\Q)}^{\GL(2g,\Q)}L_{\GL}^{(\lambda)}$ has an $\Sp$-irreducible component $L^{[1^k]}_{\Sp}$,
then a partition $\lambda$ is either $\lambda=(2^2,1^{k-2})$, $(2,1^k)$ or $(1^{k+2})$. We have
\begin{eqnarray*}
 \ [L_{\GL}^{(1^{k+2})}:H_{\Q} \otimes_\Q \mathcal{L}_{2g}^{\Q}(k+1)]&=&[L_{\GL}^{(1^{k+1})}:\mathcal{L}_{2g}^{\Q}(k+1)]=0, \\
 \ [L_{\GL}^{(1^{k+2})}:\mathcal{L}_{2g}^{\Q}(k+2)]&=&0, \\
 \ [L_{\GL}^{(2,1^k)}:H_{\Q} \otimes_\Q \mathcal{L}_{2g}^{\Q}(k+1)]&=&[L_{\GL}^{(1^{k+1})}:\mathcal{L}_{2g}^{\Q}(k+1)]
       +[L_{\GL}^{(2,1^{k-1})}:\mathcal{L}_{2g}^{\Q}(k+1)]=1, \\
 \ [L_{\GL}^{(2,1^k)}:\mathcal{L}_{2g}^{\Q}(k+2)]&=&1, \\
 \ [L_{\Sp}^{[1^k]}:\mathcal{C}_{2g}^{\Q}(k)]&=&[L_{\GL}^{(1^k)}:\mathcal{C}_{2g}^{\Q}(k)]
     =\begin{cases} 1 \hspace{1em} & \mathrm{if} \hspace{1em} k : \mathrm{odd}, \\
                                                   0 & \mathrm{if} \hspace{1em} k : \mathrm{even}. \end{cases}
\end{eqnarray*}

Suppose $k \equiv 1,3 \ (\mathrm{mod} \, 4)$. Then
\begin{eqnarray*}
 \ [L_{\GL}^{(2^2,1^{k-2})}:H_{\Q} \otimes_\Q \mathcal{L}_{2g}^{\Q}(k+1)]
    &=&[L_{\GL}^{(2^2,1^{k-3})}:\mathcal{L}_{2g}^{\Q}(k+1)]+[L_{\GL}^{(2,1^{k-1})}:\mathcal{L}_{2g}^{\Q}(k+1)], \\
    &=& \begin{cases} \frac{k-1}{2} \hspace{1em} & \mathrm{if} \hspace{1em} k \equiv 3 \pmod{4}, \\
                      \frac{k+1}{2} & \mathrm{if} \hspace{1em} k \equiv 1 \pmod{4}, \end{cases} \\
 \ [L_{\GL}^{(2^2,1^{k-2})}:\mathcal{L}_{2g}^{\Q}(k+2)]&=& \dfrac{k-1}{2}.
\end{eqnarray*}

Suppose $k \equiv 0,2 \ (\mathrm{mod} \, 4)$. Then
\begin{eqnarray*}
 \ [L_{\GL}^{(2^2,1^{k-2})}:H_{\Q} \otimes_\Q \mathcal{L}_{2g}^{\Q}(k+1)]
    &=&[L_{\GL}^{(2^2,1^{k-3})}:\mathcal{L}_{2g}^{\Q}(k+1)]+[L_{\GL}^{(2,1^{k-1})}:\mathcal{L}_{2g}^{\Q}(k+1)], \\
    &=& \dfrac{k}{2}, \\
 \ [L_{\GL}^{(2^2,1^{k-2})}:\mathcal{L}_{2g}^{\Q}(k+2)]
    &=& \begin{cases} \frac{k-2}{2} \hspace{1em} & \mathrm{if} \hspace{1em} k \equiv 2 \pmod{4}, \\
                      \frac{k}{2}                & \mathrm{if} \hspace{1em} k \equiv 0 \pmod{4}. \end{cases}
\end{eqnarray*}
Hence we obtain the claim.
\end{enumerate}
\end{proof}

\begin{rem}
By the argument above, the $\Sp$-irreducible component $[1^k]_{\Sp}$ appears in the restriction of the $\GL$-irreducible component $(2^2,1^{k-2})_{\GL}$.
\end{rem}

\begin{rem}
Our calculation above gives a combinatorial description of the $\GL$ (and $\Sp$) irreducible decomposition of $\mf{h}_{g,1}^\Q$
obtained by Kontsevich in \cite{Kon1} and \cite{Kon2}.
\end{rem}

\begin{rem}
In \cite{NT}, Nakamura and Tsunogai completely calculated $\Sp$-irreducible
decompositions of $\mf{h}_{g,1}(k)$ for $1 \le k \le 15$. In their
table, we can check that $\Sp$-irreducible components $[1^k]$ have
multiplicity one for $k=5,9,13$ and $k=6,10,14$.
\end{rem}

\subsection{Descriptions of maximal vectors}

To give an explicit description of maximal vectors, we use an $(i,j)$-expansion operator $D_{ij}:H_{\Q}^{\otimes{k}} \to H_{\Q}^{\otimes(k+2)}$ defined by 
{\small
\[
(v_1\otimes v_2 \otimes \cdots \otimes v_k) \cdot D_{ij}:=\sum_{r=1}^{2g}v_1 \otimes \cdots \otimes v_{i-1} \otimes e_r \otimes v_i \otimes \cdots \otimes v_{j-2} \otimes e_r^* \otimes v_{j-1} \otimes \cdots \otimes v_k
\]
}
for $1 \le i<j \le k+2$. Using this, we obtain several maximal vectors satisfying the condition of Proposition \ref{prop:coker}.
First we consider a maximal vector which defines the Morita obstruction $[k]$ in $\mathrm{Coker}(\Im \tau_{k,\Q}^{\M})$.
\begin{thm}[Morita and Nakamura]\label{mt1}
Let $k$ be an odd integer such that $k \geq 3$. Suppose $g\ge k+2$. An element 
\begin{eqnarray*}
\varphi_{[k]}&:=&(\omega \otimes e_1^{\otimes{k}}) \cdot \theta_P \cdot (1+\sigma_{k+2}+ \cdots +\sigma_{k+2}^{k+1})\\
&=&2\left(\sum_{i=1}^{k+1}\sum_{r=1}^{k-i+2}(-1)^{r-1}{}_{k}C_{r-1}(e_1^{\otimes{k}}) \cdot D_{i,i+r}\right).
\end{eqnarray*}
is a maximal vector with highest weight $[k]$ in $\mf{h}_{g,1}^{\Q}(k)$. Moreover this gives a unique irreducible component of $[k]$ in $\Coker{\tau_{k,\Q}^{\M}}$.
\end{thm}

This fact was originally showed by Morita and Nakamura. More precisely, Morita \cite{Mo1} showed that $[k]$ appears in $\h_{g,1}(k)$ for odd $k \geq 3$
with multiplicity at least one, using the Morita trace map. Nakamura showed that the multiplicity of $[k]$ in $\h_{g,1}(k)$ for odd $k \geq 3$ is exactly one,
and determined the maximal vector with highest weight $[k]$ in his unpublished work.

\vspace{0.5em}

Second we consider a maximal vector which defines the $\Sp(2g,\Q)$-module with highest weight $[1^k]$
in $\mathrm{Coker}(\Im \tau_{k,\Q}^{\M})$ for $k \equiv 1 \, (\mathrm{mod} \, {4})$ and $k \geq 5$.
\begin{thm}\label{mt2}
Suppose $k \equiv 1 \ (\mathrm{mod} \, 4)$, $k \geq 5$ and $g \ge k+2$. An element 
\begin{eqnarray*}
\varphi_{[1^k]}&:=&(\omega \otimes (e_1\wedge \cdots \wedge e_k)) \cdot \theta_P \cdot (1+\sigma_{k+2}+ \cdots +\sigma_{k+2}^{k+1})\\
&=&2\left(\sum_{i=1}^{k+1}\sum_{r=1}^{k-i+2}(-1)^{\delta_{r \equiv 2,3 \, (\mathrm{mod} \, 4)}}{}_{\frac{k-1}{2}}C_{\lfloor\frac{r-1}{2}\rfloor} (e_1 \wedge \cdots \wedge e_k)\cdot D_{i,i+r}\right)
\end{eqnarray*}
is a maximal vector with highest weight $[1^k]$ in $\mf{h}_{g,1}^{\Q}(k)$. Moreover this gives a unique irreducible component of $[1^k]$ in $\Coker{\tau_{k,\Q}^{\M}}$.
\end{thm}
\subsection{Proofs of main theorems}
We will give proofs of Theorem \ref{mt1} and Theorem \ref{mt2}. But, since our proof for Theorem \ref{mt1} is easier than that of Theorem \ref{mt2}, we omit the details for Theorem \ref{mt1}. 

\subsubsection{Proof of Theorem \ref{mt2}}
\textbf{Step.1} \ For $r \equiv 2 \pmod{4}$, we prove 
\begin{eqnarray*}
&{}& (e_1 \wedge \cdots \wedge e_k)D_{12}(1-s_2)(1-s_3s_2) \cdots (1-s_r \cdots s_3s_2)\\
&=& \sum_{j=1}^{r}(-1)^{\delta_{j \equiv 2,3 \, (\mathrm{mod} \, 4)}}{}_{\frac{r-2}{2}}C_{\lfloor\frac{j-1}{2}\rfloor}(e_1 \wedge \cdots \wedge e_k)D_{1,1+j}
\end{eqnarray*}
by the induction on $r$. \\
\quad Indeed, if $p=2$, the both side of the formula above coincide with $(e_1 \wedge \cdots \wedge e_k)(D_{12}-D_{13})$. 
Suppose $p>2$ and $p+4 \le k+1$. For simplicity we denote $(e_1 \wedge \cdots \wedge e_k)D_{ij}$ by $D_{i,j}^{\mathrm{sgn}}$. We have
\begin{eqnarray*}
&{}&D^{\mathrm{sgn}}_{1,1+j}(1-s_{p+1} \cdots s_2)(1-s_{p+2} \cdots s_2)(1-s_{p+3} \cdots s_2)(1-s_{p+4} \cdots s_2) \\
&=&(D^{\mathrm{sgn}}_{1,1+j}-(-1)^{p+1}D^{\mathrm{sgn}}_{1,2+j})(1-s_{p+2} \cdots s_2)(1-s_{p+3} \cdots s_2)(1-s_{p+4} \cdots s_2) \\
&\stackrel{p:even}{=}&(D^{\mathrm{sgn}}_{1,1+j}+D^{\mathrm{sgn}}_{1,2+j})(1-s_{p+2} \cdots s_2)(1-s_{p+3} \cdots s_2)(1-s_{p+4} \cdots s_2) \\
&=&(D^{\mathrm{sgn}}_{1,1+j}+D^{\mathrm{sgn}}_{1,2+j}-(-1)^{p+2}D^{\mathrm{sgn}}_{1,2+j}-(-1)^{p+2}D^{\mathrm{sgn}}_{1,3+j})(1-s_{p+3} \cdots s_2)(1-s_{p+4} \cdots s_2) \\
&\stackrel{p:even}{=}&(D^{\mathrm{sgn}}_{1,1+j}-D^{\mathrm{sgn}}_{1,3+j})(1-s_{p+3} \cdots s_2)(1-s_{p+4} \cdots s_2) \\
&=&(D^{\mathrm{sgn}}_{1,1+j}-D^{\mathrm{sgn}}_{1,3+j}-(-1)^{p+3}D^{\mathrm{sgn}}_{1,2+j}+(-1)^{p+3}D^{\mathrm{sgn}}_{1,4+j})(1-s_{p+4} \cdots s_2) \\
&\stackrel{p:even}{=}&(D^{\mathrm{sgn}}_{1,1+j}+D^{\mathrm{sgn}}_{1,2+j}-D^{\mathrm{sgn}}_{1,3+j}-D^{\mathrm{sgn}}_{1,4+j})(1-s_{p+4} \cdots s_2) \\
&=&D^{\mathrm{sgn}}_{1,1+j}+D^{\mathrm{sgn}}_{1,2+j}-D^{\mathrm{sgn}}_{1,3+j}-D^{\mathrm{sgn}}_{1,4+j}-(-1)^{p+4}(D^{\mathrm{sgn}}_{1,2+j}+D^{\mathrm{sgn}}_{1,3+j}-D^{\mathrm{sgn}}_{1,4+j}-D^{\mathrm{sgn}}_{1,5+j}) \\
&\stackrel{p:even}{=}&D^{\mathrm{sgn}}_{1,1+j}-2D^{\mathrm{sgn}}_{1,3+j}+D^{\mathrm{sgn}}_{1,5+j}.
\end{eqnarray*}
Therefore, the action of $(1-s_{p+1} \cdots s_2)(1-s_{p+2} \cdots s_2)(1-s_{p+3} \cdots s_2)(1-s_{p+4} \cdots s_2)$ on $\displaystyle \sum_{j=1}^{r}(-1)^{\delta_{j \equiv 2,3 \, (\mathrm{mod} \, 4)}}{}_{\frac{r-2}{2}}C_{\lfloor\frac{j-1}{2}\rfloor}D^{\mathrm{sgn}}_{1,1+j}$ is obtained by the following way:
\begin{eqnarray*}
&{}&\sum_{j=1}^{r}(-1)^{\delta_{j \equiv 2,3 \, (\mathrm{mod} \, 4)}}{}_{\frac{r-2}{2}}C_{\lfloor\frac{j-1}{2}\rfloor}(D^{\mathrm{sgn}}_{1,1+j}-2D^{\mathrm{sgn}}_{1,3+j}+D^{\mathrm{sgn}}_{1,5+j}) \\
&=&\sum_{j=5}^{p}\left\{
(-1)^{\delta_{j \equiv 2,3 \, (\mathrm{mod} \, 4)}}{}_{\frac{p-2}{2}}C_{\lfloor \frac{j-1}{2}\rfloor}
-2(-1)^{\delta_{j \equiv 0,1 \, (\mathrm{mod} \, 4)}}{}_{\frac{p-2}{2}}C_{\lfloor \frac{j-3}{2}\rfloor}
+(-1)^{\delta_{j \equiv 2,3 \, (\mathrm{mod} \, 4)}}{}_{\frac{p-2}{2}}C_{\lfloor \frac{j-5}{2}\rfloor}
\right\}D_{1,1+j}^{\mathrm{sgn}} \\
&{}&+D^{\mathrm{sgn}}_{12}-D^{\mathrm{sgn}}_{13}-\dfrac{p-2}{2}D^{\mathrm{sgn}}_{14}+\dfrac{p-2}{2}D^{\mathrm{sgn}}_{15}-2(D^{\mathrm{sgn}}_{14}-D^{\mathrm{sgn}}_{15}+D^{\mathrm{sgn}}_{1,p+2}-D^{\mathrm{sgn}}_{1,p+3})\\
&{}&-\dfrac{p-2}{2}D^{\mathrm{sgn}}_{1,p+2}+\dfrac{p-2}{2}D^{\mathrm{sgn}}_{1,p+3}+D^{\mathrm{sgn}}_{1,p+4}-D^{\mathrm{sgn}}_{1,p+5} \\
&=&\sum_{j=1}^{p+4}(-1)^{\delta_{j \equiv 2,3 \, (\mathrm{mod} \, 4)}}{}_{\frac{p+2}{2}}C_{\lfloor\frac{j-1}{2}\rfloor}D^{\mathrm{sgn}}_{1,1+j}.
\end{eqnarray*}
\textbf{Step.2} \ We have 
\begin{eqnarray*}
&{}&(e_1\wedge \cdots \wedge e_k)D_{ij}s_{k+1} \cdots s_2s_1\\
&{}& \hspace{3em} = \left\{
\begin{array}{ll}
(e_1\wedge \cdots \wedge e_k)(-1)^{k-1}D_{i+1,j+1}\stackrel{k:even}{=}(e_1\wedge \cdots \wedge e_k)D_{i+1,j+1} \hspace{1em} & \mathrm{if} \hspace{1em} j \neq k+2, \\
-(e_1\wedge \cdots \wedge e_k)D_{1,i+1} \hspace{1em} & \mathrm{if} \hspace{1em} j=k+2
\end{array}
\right.
\end{eqnarray*}
for $k \equiv 1 \ (\mathrm{mod} \, 4)$. Hence we obtain an explicit formula 
\begin{eqnarray*}
&{}&(\omega \otimes (e_1\wedge \cdots \wedge e_k)) \cdot \theta_P \cdot (1+\sigma_{k+2}+ \cdots +\sigma_{k+2}^{k+1})\\
&{}& \hspace{3em} = 2\sum_{i=1}^{k+1}\sum_{j=1}^{k-i+2}(-1)^{\delta_{j \equiv 2,3 \, (\mathrm{mod} \, 4)}}{}_{\frac{k-1}{2}}C_{\lfloor\frac{j-1}{2}\rfloor}
       (e_1 \wedge \cdots \wedge e_k)\cdot D_{i,i+j}.
\end{eqnarray*}
In fact, 
\begin{eqnarray*}
&{}&\sum_{j=1}^{k+1}(-1)^{\delta_{j \equiv 2,3 \, (\mathrm{mod} \, 4)}}{}_{\frac{k-1}{2}}C_{\lfloor\frac{j-1}{2}\rfloor}D^{\mathrm{sgn}}_{1,1+j}
  (1+\sigma_{k+2}+ \cdots +\sigma_{k+2}^{k+1}) \\
&=&\sum_{j=1}^{k+1}(-1)^{\delta_{j \equiv 2,3 \, (\mathrm{mod} \, 4)}}{}_{\frac{k-1}{2}}C_{\lfloor\frac{j-1}{2}\rfloor}\left(\sum_{i=1}^{k+2-j}D^{\mathrm{sgn}}_{i,i+j}-\sum_{i=1}^{j}D^{\mathrm{sgn}}_{i,i+k+2-j}\right) \\
&=&\sum_{i=1}^{k+1}\sum_{j=1}^{k+2-i}(-1)^{\delta_{j \equiv 2,3 \, (\mathrm{mod} \, 4)}}{}_{\frac{k-1}{2}}C_{\lfloor\frac{j-1}{2}\rfloor}D_{i,i+j}^{\mathrm{sgn}}-\sum_{i=1}^{k+1}
   \sum_{j=i}^{k+1}(-1)^{\delta_{j \equiv 2,3 \, (\mathrm{mod} \, 4)}}{}_{\frac{k-1}{2}}C_{\lfloor\frac{j-1}{2}\rfloor}D_{i,i+k+2-j}^{\mathrm{sgn}} \\
&=&\sum_{i=1}^{k+1}\sum_{j=1}^{k+2-i}(-1)^{\delta_{j \equiv 2,3 \, (\mathrm{mod} \, 4)}}{}_{\frac{k-1}{2}}C_{\lfloor\frac{j-1}{2}\rfloor}D_{i,i+j}^{\mathrm{sgn}}-\sum_{i=1}^{k+1}
   \sum_{j=1}^{k+2-i}(-1)^{\delta_{k+2-j \equiv 2,3 \, (\mathrm{mod} \, 4)}}{}_{\frac{k-1}{2}}C_{\lfloor\frac{k+1-j}{2}\rfloor}D_{i,i+j}^{\mathrm{sgn}} \\
&=&\sum_{i=1}^{k+1}\sum_{j=1}^{k+2-i}(-1)^{\delta_{j \equiv 2,3 \, (\mathrm{mod} \, 4)}}{}_{\frac{k-1}{2}}C_{\lfloor\frac{j-1}{2}\rfloor}D_{i,i+j}^{\mathrm{sgn}}+\sum_{i=1}^{k+1}
   \sum_{j=1}^{k+2-i}(-1)^{\delta_{j \equiv 2,3 \, (\mathrm{mod} \, 4)}}{}_{\frac{k-1}{2}}C_{\lfloor\frac{k+1-j}{2}\rfloor}D_{i,i+j}^{\mathrm{sgn}} \\
&=&2\sum_{i=1}^{k+1}\sum_{j=1}^{k+2-i}(-1)^{\delta_{j \equiv 2,3 \, (\mathrm{mod} \, 4)}}{}_{\frac{k-1}{2}}C_{\lfloor\frac{j-1}{2}\rfloor}D_{i,i+j}^{\mathrm{sgn}}.
\end{eqnarray*}
\textbf{Step.3} \ Let us consider a surjective $\Sp$-homomorphism
\[
\cont_k:H_{\Q}^{\otimes (k+2)} \isoto H_{\Q}^* \otimes H_{\Q}^{\otimes (k+1)} \twoheadrightarrow H^{\otimes{k}}_\Q
\]
by composing an $\Sp$-isomorphism $H_{\Q}^{\otimes (k+2)} \rightarrow H_{\Q}^* \otimes H_{\Q}^{\otimes (k+1)}$ induced from
$H_\Q \isoto H^*_\Q$ given by (\ref{isomH}) and a contraction homomorphism. Then we obtain 
\[
\cont_k((e_1\wedge \cdots \wedge e_k)D_{ij})=\left\{
\begin{array}{ll}
(-2g)(e_1\wedge \cdots \wedge e_k) \hspace{1em} & \mathrm{if} \hspace{1em} i=1, \,\, j=2, \\
(-1)^{j-2}(e_1\wedge \cdots \wedge e_k) \hspace{1em} & \mathrm{if} \hspace{1em} i=1, \,\, j \ge 3, \\
(-1)^{j-3}(e_1\wedge \cdots \wedge e_k) \hspace{1em} & \mathrm{if} \hspace{1em} i=2, \,\, j \ge 3, \\
0 \hspace{1em} & \mathrm{if} \hspace{1em} \text{otherwise}.
\end{array}
\right.
\]
To prove these formulae, let us recall that 
\[
\langle e_i,e_j\rangle=0=\langle e_{i'},e_{j'}\rangle,\quad 
\langle e_i,e_{j'}\rangle=\delta_{ij}=-\langle e_{j'},e_{i}\rangle, \quad (1 \le i \le g). 
\]
and 
\[
e_i^*=\left\{
\begin{array}{ll}
e_{i'}, & (1 \le i \le g), \\
-e_{i'}, & (g+1 \le i \le 2g).
\end{array}
\right.
\]
where $i':=2g-i+1$ for each integer $1 \le i \le 2g$. \\
\quad Then we have
\begin{eqnarray*}
\cont_k(D_{12}^{\sgn})&=&\cont_k\left(\sum_{r=1}^{2g}e_r \otimes e_r^* \otimes (e_1 \wedge \cdots \wedge e_k)\right) \\
&=&\sum_{r=1}^{2g}\langle e_r^*,e_r \rangle e_1 \wedge \cdots \wedge e_k=(-2g)e_1 \wedge \cdots \wedge e_k.
\end{eqnarray*}
Moreover, 
\begin{eqnarray*}
\cont_k(D_{1j}^{\sgn})&=&\cont_k\left(\sum_{r=1}^{2g}\sum_{\sigma \in \mf{S}_k}\sgn(\sigma)e_r \otimes e_{\sigma(1)} \otimes e_{\sigma(2)} \otimes \cdots \otimes \stackrel{\stackrel{j}{\vee}}{e_r^*} \otimes \cdots \otimes e_{\sigma(k)}\right) \\
&=&\sum_{r=1}^{2g}\sum_{\sigma \in \mf{S}_k}\sgn(\sigma)\langle e_{\sigma(1)},e_r\rangle \otimes e_{\sigma(2)} \otimes \cdots \otimes \stackrel{\stackrel{j-2}{\vee}}{e_r^*} \otimes \cdots \otimes e_{\sigma(k)} \\
&=&\sum_{\sigma \in \mf{S}_k}\sgn(\sigma)e_{\sigma(2)} \otimes \cdots \otimes \stackrel{\stackrel{j-2}{\vee}}{e_{\sigma(1)'}^*} \otimes \cdots \otimes e_{\sigma(k)}\\
&=&-\sum_{\sigma \in \mf{S}_k}\sgn(\sigma)e_{\sigma(2)} \otimes \cdots \otimes \stackrel{\stackrel{j-2}{\vee}}{e_{\sigma(1)}} \otimes \cdots \otimes e_{\sigma(k)} \\
&=&-(e_1 \wedge \cdots \wedge e_k)\cdot s_1s_2 \cdots s_{j-3} \\
&=&(-1)^{j-2}e_1 \wedge \cdots \wedge e_k,
\end{eqnarray*}
and similarly,
\begin{eqnarray*}
\cont_k(D_{2j}^{\sgn})&=&\cont_k\left(\sum_{r=1}^{2g}\sum_{\sigma \in \mf{S}_k}\sgn(\sigma)e_{\sigma(1)} \otimes e_r \otimes e_{\sigma(2)} \otimes \cdots \otimes \stackrel{\stackrel{j}{\vee}}{e_r^*} \otimes \cdots \otimes e_{\sigma(k)}\right) \\
&=&\sum_{r=1}^{2g}\sum_{\sigma \in \mf{S}_k}\sgn(\sigma)\langle e_r,e_{\sigma(1)}\rangle \otimes e_{\sigma(2)} \otimes \cdots \otimes \stackrel{\stackrel{j-2}{\vee}}{e_r^*} \otimes \cdots \otimes e_{\sigma(k)} \\
&=&\sum_{\sigma \in \mf{S}_k}\sgn(\sigma)e_{\sigma(2)} \otimes \cdots \otimes \stackrel{\stackrel{j-2}{\vee}}{e_{\sigma(1)'}^*} \otimes \cdots \otimes e_{\sigma(k)}\\
&=&\sum_{\sigma \in \mf{S}_k}\sgn(\sigma)e_{\sigma(2)} \otimes \cdots \otimes \stackrel{\stackrel{j-2}{\vee}}{e_{\sigma(1)}} \otimes \cdots \otimes e_{\sigma(k)} \\
&=&(e_1 \wedge \cdots \wedge e_k)\cdot s_1s_2 \cdots s_{j-3} \\
&=&(-1)^{j-3}e_1 \wedge \cdots \wedge e_k.
\end{eqnarray*}
For $i \ge 3$, because of $g>k$, it is clear that $\cont_k((e_1 \wedge \cdots \wedge e_k)D_{ij})=0$.\\
\quad {} \\
\textbf{Step.4} \ We obtain $c_k(\varphi_{[1^k]}) \neq 0$. \\
\quad Indeed, for the natural surjection $\pr:H^{\otimes{k}}_\Q \to \mathcal{C}_{2g}^{\Q}(k)$, we have
\begin{eqnarray*}
&{}&c(\varphi_{[1^k]})\\
&=&2\left(
\begin{array}{l}
\displaystyle \sum_{j=1}^{k+1}(-1)^{\delta_{j \equiv 2,3 \, (\mathrm{mod} \, 4)}}{}_{\frac{k-1}{2}}C_{\lfloor\frac{j-1}{2}\rfloor}c_k(e_1\wedge \cdots \wedge e_kD_{1,1+j}) \\
\displaystyle +\sum_{j=1}^{k}(-1)^{\delta_{j \equiv 2,3 \, (\mathrm{mod} \, 4)}}{}_{\frac{k-1}{2}}C_{\lfloor\frac{j-1}{2}\rfloor}c_k(e_1\wedge \cdots \wedge e_kD_{2,2+j})
\end{array}\right) \\
&=&2\left(
\begin{array}{l}
\displaystyle -2g
+\sum_{j=2}^{k+1}(-1)^{\delta_{j \equiv 2, 3 \, (\mathrm{mod} \, 4)}}{}_{\frac{k-1}{2}}C_{\lfloor\frac{j-1}{2}\rfloor}(-1)^{j-1}\\
\displaystyle +\sum_{j=1}^{k}(-1)^{\delta_{j \equiv 2,3 \, (\mathrm{mod} \, 4)}}{}_{\frac{k-1}{2}}C_{\lfloor\frac{j-1}{2}\rfloor}(-1)^{j-1}
\end{array}
\right)\pr(e_1 \wedge \cdots \wedge e_k) \\
&=&2\left(
-2g+2+2\sum_{j=2}^{k}(-1)^{j-1+\delta_{j \equiv 2,3 \, (\mathrm{mod} \, 4)}}{}_{\frac{k-1}{2}}C_{\lfloor\frac{j-1}{2}\rfloor}
\right)\pr(e_1 \wedge \cdots \wedge e_k) \\
&=&2\left(
-2g-2+2\sum_{j=1}^{k+1}(-1)^{j-1+\delta_{j \equiv 2,3 \, (\mathrm{mod} \, 4)}}{}_{\frac{k-1}{2}}C_{\lfloor\frac{j-1}{2}\rfloor}
\right)\pr(e_1 \wedge \cdots \wedge e_k).
\end{eqnarray*}
Here, we claim that 
\[
\sum_{j=1}^{k+1}(-1)^{j-1+\delta_{j \equiv 2,3 \, (\mathrm{mod} \, 4)}}{}_{\frac{k-1}{2}}C_{\lfloor\frac{j-1}{2}\rfloor}=0.
\]
In fact, by setting $k=4K+1$, we have
\begin{eqnarray*}
&{}&\sum_{j=1}^{k+1}(-1)^{j-1+\delta_{j \equiv 2,3 \, (\mathrm{mod} \, 4)}}{}_{\frac{k-1}{2}}C_{\lfloor\frac{j-1}{2}\rfloor}\\
&=&\sum_{j=1}^{k+1}(-1)^{\delta_{j \equiv 0,3 \, (\mathrm{mod} \, 4)}}{}_{\frac{k-1}{2}}C_{\lfloor\frac{j-1}{2}\rfloor} \\
&=&\sum_{\substack{1 \le j \le k+1 \\[1pt] j : \mathrm{odd}}}(-1)^{\delta_{j \equiv 3 \, (\mathrm{mod} \, 4)}}{}_{\frac{k-1}{2}}C_{\lfloor\frac{j-1}{2}\rfloor}
  +\sum_{\substack{1 \le j \le k+1 \\[1pt] j : \mathrm{even}}}(-1)^{\delta_{j \equiv 0 \, (\mathrm{mod} \, 4)}}{}_{\frac{k-1}{2}}C_{\lfloor\frac{j-1}{2}\rfloor} \\
&=&\sum_{p=0}^{2K}(-1)^{\delta_{p \equiv 1 \, (\mathrm{mod} \, 2)}}{}_{2K}C_p+\sum_{q=1}^{2K+1}(-1)^{\delta_{q \equiv 0 \, (\mathrm{mod} \, 2)}}{}_{2K}C_{q-1} \\
&=&2\sum_{p=0}^{2K}(-1)^{\delta_{p \equiv 1 \, (\mathrm{mod} \, 2)}}{}_{2K}C_p=2(1-1)^{2K}=0.
\end{eqnarray*}
Hence, we conclude $c_k(\varphi_{[1^k]})=-4(g+1)\pr(e_1 \wedge \cdots \wedge e_k)$. \\
\quad Since $[L^{[1^k]}:H^{\otimes{k}}_{\Q}]=[L^{[1^k]}:\mathcal{C}_{2g}^{\Q}(k)]=1$ and $e_1 \wedge \cdots \wedge e_k$ is a maximal vector
with highest weight $(1^k)$ of $H_\Q^{\otimes{k}}$, we have $\pr(e_1 \wedge \cdots \wedge e_k) \neq 0$. \\
\quad {} \\
\textbf{Step.5} \ By Proposition \ref{prop:coker} and Proposition \ref{prop:mult}, the maximal vector $\varphi_{[1^k]}$ gives a unique irreducible component of $[1^k]$
in $\Coker{\tau_{k,\Q}^{\M}}$.

\vspace{0.5em}

This completes the proof of Theorem \ref{mt2}. \qed

\subsubsection{Outline of proof of Theorem \ref{mt1}}
To begin with, we can show
\[
(e_1^{\otimes{k}}D_{12})(1-s_2)(1-s_3s_2) \cdots (1-s_r \cdots s_3s_2)=\sum_{j=1}^{r}(-1)^{j-1}{}_{r-1}C_{j-1}(e_1^{\otimes{k}})D_{1,1+j}
\]
by using the induction on $r$.
Secondly, we have 
\[
(e_1^{\otimes{k}}D_{ij})s_{k+1}s_k \cdots s_2s_1=\left\{
\begin{array}{ll}
e_1^{\otimes{k}}D_{i+1,j+1}, \hspace{1em} & \mathrm{if} \,\,\, j \neq k+2, \\
-e_1^{\otimes{k}}D_{1,i+1}, & \mathrm{if} \,\,\, j=k+2.
\end{array}
\right.
\]
Hence we get an explicit formula
\[
(\omega \otimes e_1^{\otimes{k}}) \cdot \theta_P \cdot (1+\sigma_{k+2}+ \cdots +\sigma_{k+2}^{k+1})
=\sum_{i=1}^{k+1}\sum_{r=1}^{k-i+2}(-1)^{r-1}{}_{k}C_{r-1}(e_1^{\otimes{k}}) \cdot D_{i,i+r}.
\]
Thirdly, we have  
\[
\cont_k(e_1^{\otimes{k}}D_{ij})=\left\{
\begin{array}{ll}
(-2g)(e_1^{\otimes{k}})\hspace{1em} & \mathrm{if} \hspace{1em} i=1, \,\, j=2, \\
-(e_1^{\otimes{k}}) \hspace{1em} & \mathrm{if} \hspace{1em} i=1, \,\, j \ge 3, \\
(e_1^{\otimes{k}}) \hspace{1em} & \mathrm{if} \hspace{1em} i=2, \,\, j \ge 3, \\
0 \hspace{1em} & \mathrm{if} \hspace{1em} \text{otherwise},
\end{array}
\right.
\]
and $\pr(e_1^{\otimes{k}}) \neq 0$. Thus we obtain
\begin{eqnarray*}
c_k(\varphi_{[k]})&=&\sum_{j=1}^{k+1}(-1)^{j-1}{}_kC_{j-1} \ c_k(e_1^{\otimes{k}}D_{1j})+\sum_{j=1}^{k}(-1)^{j-1}{}_kC_{j-1} \ c_k(e_1^{\otimes{k}}D_{2j}) \\
&=&\left(-2g-\sum_{j=2}^{k+1}(-1)^{j-1}{}_kC_{j-1}+\sum_{j=1}^{k}(-1)^{j-1}{}_kC_{j-1}\right)\pr(e_1^{\otimes{k}}) \\
&=&\left(-2g+(-1)^{k+1}+\sum_{j=2}^k\left\{(-1)^j{}_kC_{j}+(-1)^{j-1}{}_kC_{j-1}\right\}+1\right)\pr(e_1^{\otimes{k}}) \\
&=&(2-2g)\pr(e_1^{\otimes{k}}) \neq 0.
\end{eqnarray*}
Therefore, by Proposition \ref{prop:coker} and Proposition \ref{prop:mult}, the maximal vector $\varphi_{[k]}$ gives
a unique irreducible component of $[k]$ in $\Coker{\tau_{k,\Q}^{\M}}$.

\vspace{0.5em}

This completes the proof of Theorem \ref{mt1}. \qed

\subsection{Problems for the Johnson cokernels}

Finally, we conclude by suggesting a problem for the Johnson cokernels of the mapping class group.

\vspace{0.5em}

By observing the table of $\mathrm{Coker}(\tau_{k,\Q}^{\M})$ for $1 \leq k \leq 4$ in Subsection {\rmfamily \ref{Ss-John}},
we see that $\mathrm{Coker}(\tau_{k,\Q}^{\M}) \cong \mathrm{Im}(c_k)$ for $1 \leq k \leq 4$ as an $\Sp(2g,\Q)$-module, where
$c_k : \h_{g,1}^{\Q}(k) \rightarrow \mathcal{C}_{2g}^{\Q}(k)$ is an $\Sp(2g,\Q)$-equivariant homomorphism defined in Subsection {\rmfamily \ref{sec:str}}.

\vspace{0.5em}

These facts let us motivate to determine whether $\mathrm{Coker}(\tau_{k,\Q}^{\M}) \cong \mathrm{Im}(c_k)$ for any $k \geq 1$, or not.
This is, however, not correct in general. In fact, for $k=6$
according to the description in \cite{Mo}, the Sp-invariant part of 
$\h_{g,1}(6)/\mathrm{Im}(\tau_{6,\Q}^{\M})$ is $\Q^3$.
On the other hand, that of $\mathcal{C}_{2g}^{\Q}(6)$ is $\Q^{\oplus 2}$. Hence we can not detect all of the $\Sp$-invariant part of $\h_{g,1}(6)$
using the map $c_6$. We have heard from Morita about these facts in his thourhtful e-mail. 

\vspace{0.5em}

Here we suggest a problem to determine the Sp-component of $\h_{g,1}(k)$ which can be detect the map $c_k$.
Namely,
\begin{prob}\label{Prob-1}
For any $k \geq 1$, determine the image $\mathrm{Im}(c_k)$ of $c_k$.
\end{prob}

\vspace{0.5em}

Let us consider a sequence of Sp-submodules of $\h_{g,1}^{\Q}$:
\[ \mathrm{Im}(\tau_{k,\Q}^{\M}) \subset \mathrm{Ker}(c_k) \subset \h_{g,1}^{\Q} \]
for each $k \geq 2$.
Problem {\rmfamily \ref{Prob-1}} is equivalent to a problem to detremine the Sp-module structure of the quotient $\h_{g,1}^{\Q}/\mathrm{Ker}(c_k)$.
We remark that from the description in \cite{Mo} as above,
for $k=6$, an irreducible module $[0]$ appears in $\mathrm{Ker}(c_6)/\mathrm{Im}(\tau_{6,\Q}^{\M})$ with multiplicity at least one.
(Morita told us this fact in his e-mail to us.) This shows $\mathrm{Im}(\tau_{k,\Q}^{\M}) \neq \mathrm{Ker}(c_k)$ in general.

\vspace{0.5em}

Let $(\h_{g,1}^{\Q})^{\mathrm{ab}}$ be the abelianization of $\h_{g,1}^{\Q}$ as a Lie algebra, and 
$[\h_{g,1}^{\Q}, \h_{g,1}^{\Q}]$ the kernel of the abelianization $\h_{g,1}^{\Q} \rightarrow (\h_{g,1}^{\Q})^{\mathrm{ab}}$.
We write $[\h_{g,1}^{\Q}, \h_{g,1}^{\Q}](k)$ for the degree $k$-part of $[\h_{g,1}^{\Q}, \h_{g,1}^{\Q}]$.
It is still open problem to determine the $\Sp$-module structure of $(\h_{g,1}^{\Q})^{\mathrm{ab}}$.
From Hain's result, see Theorem {\rmfamily \ref{T-Hain}}, we have
\[ \mathrm{Im}(\tau_{k,\Q}^{\M}) \subset [\h_{g,1}^{\Q}, \h_{g,1}^{\Q}](k) \subset \h_{g,1}^{\Q} \]
for each $k \geq 2$.
In \cite{Mo}, Morita constructed a surjective Lie algebra homomorphism
\[ \tau_{1,\Q}^{\M} \oplus \bigoplus_{k \geq 1} \mathrm{Tr}_{2k+1} : \h_{g,1}^{\Q} \rightarrow \Lambda^3 H_{\Q} \oplus \bigoplus_{k \geq 1} S^{2k+1} H_{\Q}  \]
using the Morita trace maps $\mathrm{Tr}_{2k+1}$, where the target is considered as an abelian Lie algebra.
Hence, the Morita obstructions can be detected by $\h_{g,1}^{\Q}(k)/\mathrm{Ker}(c_k)$ and $(\h_{g,1}^{\Q})^{\mathrm{ab}}$. 
Recently, J. Conant, M. Kassabov and K. Vogtmann announced there are new series in $(\h_{g,1}^{\Q})^{\mathrm{ab}}$
other than the Morita obstructions. 

\vspace{0.5em}

Then we have a problem:
\begin{prob}
Does there exist an irreducible Sp-module $L \subset \mathrm{Ker}(c_k)$ such that $L \not\subset [\h_{g,1}^{\Q}, \h_{g,1}^{\Q}](k)$?
For example, clarify whether or not the Conant-Kassabov-Vogtmann obstruction is contained in $\mathrm{Ker}(c_k)$.
\end{prob}

\subsection*{Acknowledgements}
\quad Both authors would like to thank Professor Shigeyuki Morita and Takuya Sakasai for valuable discussions about our results
and related topics and sincere encouragement for our research.
They would also like to thank J. Conant and M. Kassabov for the discussion about their recent works. \\
\quad They are supported by JSPS Research Fellowship for Young Scientists and the Global COE program at Kyoto University. \\
\quad In November 2004, at Okayama University, Professor Hiroaki Nakamura showed the second author (T. S.) his explicit calculation [NT], and
they discussed the multiplicities of $[1^k]$.
They checked that the multiplicities of $[1^{4k+1}]$ and
$[1^{4k+2}]$ in $\mf{h}_{g,1}(k)$ are exactly one for $1 \leq k \leq 3$. Nakamura communicated to him 
the possibilities that the multiplicities of $[1^{4k+1}]$ in $\mf{h}_{g,1}(k)$ is
exactly one for general $k$, and that they survive in the Johnson cokernels. The
second author would like to thank Professor Nakamura for these suggestions which motivated him to study the mapping class group of a surface.\\
\quad The first author (N. E.) would like to thank Kentaro Wada for his kindness guidance for dealing with idempotents and the Brauer algebras. He also would like to thank Yuichiro Hoshi for his comments on the arithmetic aspects of the mapping class groups.

\end{document}